\documentclass[12pt]{amsart}
\usepackage{}
\usepackage{amsmath}
\usepackage{amsfonts}
\usepackage{amssymb}
\usepackage{bbding}
\usepackage{stmaryrd}
\usepackage{txfonts}
\usepackage{graphicx}
\usepackage{epsfig}
\usepackage{xypic}
\usepackage{tikz}
\usepackage{longtable}
\usepackage[all]{xy}
\usepackage{pgflibraryarrows}
\usepackage{pgflibrarysnakes}
\usepackage[shortlabels]{enumitem}
\usepackage{ifpdf}
\ifpdf
  \usepackage[colorlinks,final,backref=page,hyperindex]{hyperref}
\else
  \usepackage[colorlinks,final,backref=page,hyperindex,hypertex]{hyperref}
\fi

\newtheorem{theorem}{Theorem}[section]
\newtheorem{lemma}[theorem]{Lemma}

\newtheorem{prop}[theorem]{Proposition}
\theoremstyle{definition}
\newtheorem{defn}[theorem]{Definition}

\newcommand{\nc}{\newcommand}
\newcommand{\delete}[1]{}

\nc{\tred}[1]{\textcolor{red}{#1}}
\nc{\tblue}[1]{\textcolor{blue}{#1}} \nc{\tgreen}[1]{\textcolor{green}{#1}} \nc{\tpurple}[1]{\textcolor{purple}{#1}} \nc{\btred}[1]{\textcolor{red}{\bf #1}} \nc{\btblue}[1]{\textcolor{blue}{\bf #1}} \nc{\btgreen}[1]{\textcolor{green}{\bf #1}} \nc{\btpurple}[1]{\textcolor{purple}{\bf #1}}

\renewcommand{\Bbb}{\mathbb}

\newcommand{\efootnote}[1]{}

\renewcommand{\textbf}[1]{}

\nc{\mlabel}[1]{\label{#1}}  
\nc{\mcite}[2][]{\cite[#1]{#2}}  
\nc{\mref}[1]{\ref{#1}}  
\nc{\mbibitem}[1]{\bibitem{#1}} 

\delete{
\nc{\mlabel}[1]{\label{#1}  
{\hfill \hspace{1cm}{\bf{{\ }\hfill(#1)}}}}
\nc{\mcite}[2][1]{\cite[#1]{#2}{{\bf{{\ }(#2; #1)}}}}  
\nc{\mref}[1]{\ref{#1}{{\bf{{\ }(#1)}}}}  
\nc{\mbibitem}[1]{\bibitem[\bf #1]{#1}} 
}

\renewcommand\geq{\geqslant}
\renewcommand\leq{\leqslant}

\renewcommand\bar[1]{\overline{#1}}

\nc{\into}{I}
\nc{\rbw}{\mathfrak{R}} \nc{\brp}{\mathrm{brp}} \nc{\lead}{\mathrm{Lead}} \nc{\Id}{\mathrm{Id}} \nc{\Irr}{\mathrm{Irr}}
\nc{\vx}{\sigma} \nc{\vy}{\tau} \nc{\dvx}{\sigma^{(1)}} \nc{\dvy}{\tau^{(1)}} \nc{\done}{\vep} \nc{\mcitep}[1]{\mcite{#1}} \nc{\wt}{\mathrm{wt}} \nc{\bre}[1]{|#1|} \nc{\mapmonoid}{\frakM} \nc{\disjoint}{\frakM'}
\nc{\ncpoly}[1]{\langle #1\rangle}  
\nc{\mapm}[1]{\lfloor\!|{#1}|\!\rfloor}
\nc{\diff}[1]{{}^\NC\{ #1 \}} \nc{\disj}[1]{\{{#1}\}'} \nc{\mdisj}[1]{\frakM'(#1)} \nc{\brho}{\bar{\rho}} \nc{\om}{\bar{\frakm}} \nc{\frakn}{\mathfrak n} \nc{\ddeg}[1]{^{(#1)}} \nc{\opset}{X} \nc{\genset}{{Z}} \nc{\NC}{\mathrm{{NC}}} \nc{\leaf}{\mathrm{leaf}} \nc{\twig}{\mathrm{twig}} \nc{\fe}{\mathrm{fl}} \nc{\munderline}[1]{#1} \nc{\bo}{o} \nc{\dep}{\mathrm{depth}} \nc{\ofe}{\mathrm{ofl}} \nc{\dfe}{\mathrm{dfe}} \nc{\fex}{\mathrm{fex}} \nc{\dl}{\mathrm{dlex}} \nc{\db}{\mathrm{db}} \nc{\lex}{\mathrm{lex}} \nc{\clex}{\mathrm{clex}} \nc{\dgp}{\mathrm{dgp}} \nc{\dgx}{\mathrm{dgx}} \nc{\br}{\mathrm{br}} \nc{\obd}{\mathrm{odb}} \nc{\ob}{\mathrm{ob}}
\nc{\pie}{\mathrm{PIE}}
\nc{\rbo}{\mathrm{RBO}}
\nc{\supp}{\mathcal{S}}
\nc{\nul}{\mathcal{Z}}
\nc{\ind}{\mathrm{ind}}

\nc{\bin}[2]{ (_{\stackrel{\scs{#1}}{\scs{#2}}})}  
\nc{\binc}[2]{ \left (\!\! \begin{array}{c} \scs{#1}\\
    \scs{#2} \end{array}\!\! \right )}  
\nc{\bincc}[2]{  \left ( {\scs{#1} \atop
    \vspace{-1cm}\scs{#2}} \right )}  
\nc{\bs}{\bar{S}} \nc{\cosum}{\sqsubset} \nc{\la}{\longrightarrow} \nc{\rar}{\rightarrow} \nc{\dar}{\downarrow} \nc{\dprod}{**} \nc{\dap}[1]{\downarrow \rlap{$\scriptstyle{#1}$}} \nc{\md}[1]{\bar{#1}} \nc{\uap}[1]{\uparrow \rlap{$\scriptstyle{#1}$}} \nc{\defeq}{\stackrel{\rm def}{=}} \nc{\disp}[1]{\displaystyle{#1}} \nc{\dotcup}{\ \displaystyle{\bigcup^\bullet}\ } \nc{\gzeta}{\bar{\zeta}} \nc{\hcm}{\ \hat{,}\ } \nc{\hts}{\hat{\otimes}} \nc{\barot}{{\otimes}} \nc{\free}[1]{\bar{#1}} \nc{\uni}[1]{\tilde{#1}} \nc{\hcirc}{\hat{\circ}} \nc{\leng}{\ell} \nc{\lleft}{[} \nc{\lright}{]} \nc{\lc}{\lfloor} \nc{\rc}{\rfloor}
\nc{\lb}{[} 
\nc{\rb}{]} 
\nc{\curlyl}{\left \{ \begin{array}{c} {} \\ {} \end{array}
    \right.  \!\!\!\!\!\!\!}
\nc{\curlyr}{ \!\!\!\!\!\!\!
    \left. \begin{array}{c} {} \\ {} \end{array}
    \right \} }
\nc{\longmid}{\left | \begin{array}{c} {} \\ {} \end{array}
    \right. \!\!\!\!\!\!\!}
\nc{\onetree}{\bullet} \nc{\ora}[1]{\stackrel{#1}{\rar}}
\nc{\ola}[1]{\stackrel{#1}{\la}}
\nc{\ot}{\otimes} \nc{\mot}{{{\boxtimes\,}}} \nc{\otm}{\overline{\boxtimes}} \nc{\sprod}{\bullet} \nc{\scs}[1]{\scriptstyle{#1}} \nc{\mrm}[1]{{\rm #1}} \nc{\msum}{\sum\limits}
\nc{\margin}[1]{\marginpar{\rm #1}}   
\nc{\dirlim}{\displaystyle{\lim_{\longrightarrow}}\,} \nc{\invlim}{\displaystyle{\lim_{\longleftarrow}}\,} \nc{\mvp}{\vspace{0.3cm}} \nc{\tk}{^{(k)}} \nc{\tp}{^\prime} \nc{\ttp}{^{\prime\prime}} \nc{\svp}{\vspace{2cm}} \nc{\vp}{\vspace{8cm}} \nc{\proofbegin}{\noindent{\bf Proof: }}
\nc{\proofend}{$\blacksquare$ \vspace{0.3cm}}
\nc{\modg}[1]{\!<\!\!{#1}\!\!>}
\nc{\intg}[1]{F_C(#1)} \nc{\lmodg}{\!<\!\!} \nc{\rmodg}{\!\!>\!} \nc{\cpi}{\widehat{\Pi}}
\nc{\sha}{{\mbox{\cyr X}}}  
\nc{\shap}{{\mbox{\cyrs X}}} 
\nc{\shpr}{\diamond}    
\nc{\shp}{\ast} \nc{\shplus}{\shpr^+}
\nc{\shprc}{\shpr_c}    
\nc{\msh}{\ast} \nc{\zprod}{m_0} \nc{\oprod}{m_1} \nc{\vep}{\varepsilon} \nc{\labs}{\mid\!} \nc{\rabs}{\!\mid}
\nc{\astarrow}{\overset{\raisebox{-3pt}{$\ast$}}{\rightarrow}}

\nc{\dth}{d} \nc{\mmbox}[1]{\mbox{\ #1\ }} \nc{\fp}{\mrm{FP}} \nc{\rchar}{\mrm{char}} \nc{\Fil}{\mrm{Fil}} \nc{\Mor}{Mor\xspace} \nc{\gmzvs}{gMZV\xspace} \nc{\gmzv}{gMZV\xspace} \nc{\mzv}{MZV\xspace} \nc{\mzvs}{MZVs\xspace} \nc{\Hom}{\mrm{Hom}} \nc{\id}{\mrm{id}} \nc{\im}{\mrm{im}} \nc{\incl}{\mrm{incl}} \nc{\map}{\mrm{Map}} \nc{\mchar}{\rm char} \nc{\nz}{\rm NZ}

\nc{\Alg}{\mathbf{Alg}} \nc{\Bax}{\mathbf{Bax}} \nc{\bff}{\mathbf f} \nc{\bfk}{{\bf k}} \nc{\bfone}{{\bf 1}} \nc{\bfx}{\mathbf x} \nc{\bfy}{\mathbf y}
\nc{\base}[1]{\bfone^{\otimes ({#1}+1)}} 
\nc{\Cat}{\mathbf{Cat}} \delete{}
\nc{\detail}{\marginpar{\bf More detail}
    \noindent{\bf Need more detail!}
    \svp}
\nc{\Int}{\mathbf{Int}} \nc{\Mon}{\mathbf{Mon}}
\nc{\rbtm}{{shuffle }} \nc{\rbto}{{Rota-Baxter }} \nc{\remarks}{\noindent{\bf Remarks: }} \nc{\Rings}{\mathbf{Rings}} \nc{\Sets}{\mathbf{Sets}}

\nc{\BA}{{\Bbb A}} \nc{\CC}{{\Bbb C}} \nc{\DD}{{\Bbb D}} \nc{\EE}{{\Bbb E}} \nc{\FF}{{\Bbb F}} \nc{\GG}{{\Bbb G}} \nc{\HH}{{\Bbb H}} \nc{\LL}{{\Bbb L}} \nc{\NN}{{\Bbb N}} \nc{\KK}{{\Bbb K}} \nc{\QQ}{{\Bbb Q}} \nc{\RR}{{\Bbb R}} \nc{\TT}{{\Bbb T}} \nc{\VV}{{\Bbb V}} \nc{\ZZ}{{\Bbb Z}}

\nc{\cala}{{\mathcal A}} \nc{\calc}{{\mathcal C}} \nc{\cald}{{\mathcal D}} \nc{\cale}{{\mathcal E}} \nc{\calf}{{\mathcal F}} \nc{\calg}{{\mathcal G}} \nc{\calh}{{\mathcal H}} \nc{\cali}{{\mathcal I}} \nc{\call}{{\mathcal L}} \nc{\calm}{{\mathcal M}} \nc{\caln}{{\mathcal N}} \nc{\calo}{{\mathcal O}} \nc{\calp}{{\mathcal P}} \nc{\calr}{{\mathcal R}} \nc{\cals}{{\mathcal S}} \nc{\calt}{{\mathcal T}} \nc{\calw}{{\mathcal W}} \nc{\calk}{{\mathcal K}} \nc{\calx}{{\mathcal X}}
\nc{\calz}{{\mathcal Z}}
 \nc{\CA}{\mathcal{A}}

\nc{\fraka}{{\mathfrak a}} \nc{\frakA}{{\mathfrak A}} \nc{\frakb}{{\mathfrak b}} \nc{\frakB}{{\mathfrak B}}
\nc{\frakc}{{\mathfrak c}}  \nc{\frakD}{{\mathfrak D}}
\nc{\frakd}{{\mathfrak d}}
\nc{\fraku}{{\mathfrak u}}
\nc{\frakv}{{\mathfrak v}}\nc{\frakw}{{\mathfrak w}}
\nc{\frakH}{{\mathfrak H}}
\nc{\frakr}{{\mathfrak r}}
\nc{\frakh}{{\mathfrak h}} \nc{\frakM}{{\mathfrak M}}
\nc{\frakO}{{\mathfrak O}}
\nc{\frako}{{\mathfrak o}}
\nc{\frakE}{{\mathfrak E}}
\nc{\frake}{{\mathfrak e}}
\nc{\bfrakM}{\overline{\frakM}} \nc{\frakm}{{\mathfrak m}} \nc{\frakP}{{\mathfrak P}} \nc{\frakN}{{\mathfrak N}} \nc{\frakp}{{\mathfrak p}} \nc{\frakS}{{\mathfrak S}}
\nc{\frakk}{{\mathfrak k}}
\nc{\frakx}{{\mathfrak x}}
\nc{\frakl}{{\mathfrak l}} \nc{\ox}{\bar{\frakx}} \nc{\frakX}{{\mathfrak X}} \nc{\fraky}{{\mathfrak y}} \nc\dop{\delta}
\nc{\Reduce}{{\rm Red}}

\font\cyr=wncyr10 \font\cyrs=wncyr7
\nc{\redt}[1]{\textcolor{red}{#1}}
\nc{\li}[1]{\textcolor{red}{#1}}
\nc{\sz}[1]{\textcolor{blue}{\tt sz:#1}}

\nc{\nonz}[1]{#1^\times}
\nc{\NNP}{\nonz{\NN}}
\nc{\stdint}[1]{J_{#1}}
\nc{\dualmod}[1]{#1^*}
\nc{\alghom}[1]{#1^\bullet}
\nc{\End}{\mathrm{End}}
\nc{\rng}{\mathrm{\mathcal{R}}}
\nc{\codim}{\mathrm{codim}}
\nc{\evl}{\mathrm{ev}}

\nc{\oh}{\circledast}
\nc{\ohs}{\oh}
\nc{\ohzs}{\odot}
\nc{\ohz}{\odot}
\nc{\ug}{\mathfrak{U}_h(\mathfrak{g})}
\nc{\frakg}{\mathfrak{g}}
\nc{\frakU}{\mathfrak U}
\nc{\brag}{[,]_\frakg}
\newenvironment{thmenumerate}{\leavevmode\begin{enumerate}[leftmargin=1.5em]}{\end{enumerate}}

\nc{\av}{\alpha_V}
\nc{\az}{\alpha_T}
\nc{\ah}{\alpha_h}
\nc{\ag}{\alpha_\frakg}

\nc{\zb}[1]{\textcolor{blue}{ #1}}

\begin{document}
\title[Enveloping algebras and PBW theorem for involutive Hom-Lie algebras]{Universal enveloping algebras and Poincar\'e-Birkhoff-Witt theorem for involutive Hom-Lie algebras}

\author{Li Guo}
\address{
Department of Mathematics and Computer Science, Rutgers University, Newark, NJ 07102, USA}
\email{liguo@rutgers.edu}

\author{Bin Zhang}
\address{School of Mathematics, Yangtze Center of Mathematics,
Sichuan University, Chengdu, 610064, P. R. China}
\email{zhangbin@scu.edu.cn}

\author{Shanghua Zheng}
\address{Department of Mathematics, Lanzhou University, Lanzhou, Gansu 730000, China}
\email{zheng2712801@163.com}

\hyphenpenalty=8000
\date{\today}

\begin{abstract}
A Hom-type algebra is called involutive if its Hom map is multiplicative and involutive.
In this paper, we obtain an explicit construction of the free involutive Hom-associative algebra on a Hom-module. We then apply this construction to obtain the universal enveloping algebra of an involutive Hom-Lie algebra. Finally we obtain a Poincar\'e-Birkhoff-Witt theorem for the enveloping associative algebra of an involutive Hom-Lie algebra.
\end{abstract}

\subjclass[2010]{17A30,17A50,17B35}

\keywords{Hom-Lie algebra, Hom-associative algebra, involution, universal enveloping algebra, Poincar\'e-Birkhoff-Witt theorem}

\maketitle

\tableofcontents

\hyphenpenalty=8000 \setcounter{section}{0}

\allowdisplaybreaks


\section{Introduction}\mlabel{sec:int}

Hom algebras were first introduced in the Lie algebra setting~\cite{HLS} with motivation from physics though its origin can be traced back in earlier literature such as~\cite{Hu}. Among the many Hom-generalizations of classical algebraic structures~\cite{AM,HMS,LS,Mak,MY,MS1,Sh,SB,Ya1,Ya2,Ya3,Ya4}, Hom-associative algebras were introduced to generalize the classical construction of Lie algebras from associative algebra by taking commutators to Hom-Lie algebras~\cite{MS2}.
The universal enveloping algebra of a Lie algebra is obtained as a suitable quotient of the free associative algebra, namely the tensor algebra, on the underlying module of the Lie algebra. For the lack of a suitable construction of free Hom-associative algebras, this approach did not work in the Hom setting. To get around this difficulty, D. Yau~\cite{Ya4} took a different approach where he constructed the free Hom-nonassociative algebra via weighted trees. Then the universal enveloping algebra of a Hom-Lie algebra is obtained as a suitable quotient of the free Hom-nonassociative algebra.
The next natural step is to generalize to Hom-Lie algebras the well-known Poincar\'e-Birkhoff-Witt Theorem which gives a basis of the enveloping (associative) algebra from a basis of the Lie algebra. Because of the size of the weighted trees, this construction of the enveloping algebra of a Hom-Lie algebra makes it intractable to obtain a Poincar\'e-Birkhoff-Witt type theorem for Hom-Lie algebra. See the discussion in~\cite{HMS}.

A recent paper~\cite{ZG} gave an explicit construction of free involutive Hom-associative algebras generated by a set, without having to go into the Hom-nonassociative realm. This construction makes it possible to adapt the classical theory of enveloping algebras of Lie algebras to the Hom setting, namely to construct the universal enveloping algebra of an involutive Hom-Lie algebra from free involutive Hom-associative algebra and to attempt for a Poincar\'e-Birkhoff-Witt type theorem for this enveloping algebra.

We carry out this approach in this paper. First in Section~\mref{sec:free} we modify the construction in~\cite{ZG} to obtain the free involutive Hom-associative algebra on a Hom-module. Based on this we provide in Section~\mref{sec:env} another construction of the enveloping algebra of involutive Hom-Lie algebras in addition to the one given in \cite{Ya4}.
Based on this new construction, we obtain in Section~\mref{sec:pbw} a Poincar\'e-Birkhoff-Witt theorem for an involutive Hom-Lie algebras, that is, we obtain a canonical basis for the enveloping algebra of an involutive Hom-Lie algebra. This shows in particular that the map from the Hom-Lie algebra to the enveloping algebra is injective, verifying a conjecture in~\cite{HMS} in the involutive case.

Throughout this paper, $\bfk$ denotes a commutative ring with an identity, which is further assumed to be a field whose characteristic is not two in the last section.

\section{Free involutive Hom-associative algebras}
\mlabel{sec:free}

The main purpose of this section is to construct the free involutive Hom-associative algebra on a Hom-module. In Section~\mref{ssec:freeonvect} we recall the basic concepts and give the construction of this free object, stated in Theorem~\mref{thm:profree}. The proof of this theorem is given in Section~\mref{ss:freepf}.

\subsection{Free involutive Hom-associative algebras on a Hom-module}
\mlabel{ssec:freeonvect}
We first recall the basic concepts.
\begin{defn}
\begin{enumerate}
\item
A {\bf Hom-module} is a pair $(V,\alpha_V)$ consisting of a $\bfk$-module $V$ and a linear operator $\alpha_V: V\rightarrow V$. It is also called an {\bf operated module} in~\cite{Gop}.
\item
A {\bf Hom-associative algebra} is a triple $(A,\cdot,\alpha_A)$ consisting of a $\bfk$-module $A$, a linear map $\cdot:A\ot A\rightarrow A$, called the multiplication, and a multiplicative linear operator  $\alpha_A:A\rightarrow A$ (namely $\alpha_A(x\cdot y)=\alpha_A(x)\cdot\alpha_A(y)$)  satisfying the {\bf Hom-associativity}
\begin{equation}
\alpha_A(x)\cdot(y\cdot z)=(x\cdot y)\cdot\alpha_A(z)
\mlabel{eq:homass}
\end{equation}
for all $x,y,z\in A$.
\item
A Hom-associative algebra $(A,\cdot,\alpha_A)$ (resp. Hom-module $(V,\alpha_V)$) is said to be {\bf involutive} if $\alpha_A^2=\id$ (resp. $\alpha_V^2=\id$).
\item
Let $(V,\alpha_V)$ and $(W,\alpha_W)$ be Hom-modules. A $\bfk$-linear map $f:V\to W$ is called a {\bf morphism of Hom-modules } if $f(\alpha_V(x))=\alpha_W(f(x))$ for all $x\in V$.
\item
Let $(A,\cdot,\alpha_A)$ and $(B,\ast,\alpha_B)$ be two Hom-associative algebras. A $\bfk$-linear map $f:A\rightarrow B$ is a {\bf morphism of Hom-associative algebras} if
$$f(x\cdot y)=f(x)\ast f(y)\quad\text{and}\quad f(\alpha_A(x))=\alpha_B(f(x))\quad\text{for all}\, x,y\in A.$$
\item
Let $(A,\cdot,\alpha_A)$ be a Hom-associative algebra.
A submodule $B\subseteq A$ is called a {\bf Hom-associative subalgebra} of $A$ if $B$ is closed under the multiplication $\cdot$ and $\alpha_A(B)\subseteq B$.
\mlabel{it:subalg}
\item
Let $(A,\cdot,\alpha_A)$ be a Hom-associative algebra.
A submodule $I\subseteq A$ is called a {\bf Hom-ideal} of $A$ if $x\cdot y\in I, \,y\cdot x\in I$ for all $x\in I,y\in A$, and  $\alpha_A(I)\subseteq I$.
\end{enumerate}
\mlabel{def:Homasso}
\end{defn}

Examples and studies of Hom-associative algebras can be found in~\cite{MS1,MS2} and the references therein.

Free involutive Hom-associative algebra on an involutive Hom-module can be defined from the left adjoint functor of the forgetful functor from the category of involutive Hom-associative algebras to the category of involutive Hom-modules. In concrete terms, we have

\begin{defn}
Let $(V,\alpha_V)$ be an involutive Hom-module. A {\bf free involutive Hom-associative algebra on $V$} is an involutive Hom-associative algebra $(F_{IHA}(V),\ast,\alpha_F)$
together with a morphism of Hom-modules $j_V:(V,\alpha_V)\rightarrow (F_{IHA}(V),\alpha_F)$ with the property that, for any involutive Hom-associative algebra $(A,\cdot,\alpha_A)$ together with a morphism $f:(V,\alpha_V)\rightarrow (A,\alpha_A)$ of Hom-modules, there is a unique morphism $\bar{f}:(F_{IHA}(V),\ast,\alpha_F)\rightarrow (A,\cdot,\alpha_A)$ of Hom-associative algebras such that $f=\bar{f}\circ j_V$.
\end{defn}

The well-known construction of the (non-unitary) free associative algebra on a module $V$ is the tensor algebra $T^+(V):=\oplus_{n\geq 1}V^{\ot n}$ equipped with the concatenation tensor product.
Modifying~\cite{ZG}, we will give a construction of the free involutive Hom-associative algebra on an involutive module $(V,\alpha_V)$ which take a similar form to the tensor algebra, hence called the {\bf Hom-tensor algebra} and denoted by $T^+_h(V)$.

As a $\bfk$-module, $T^+_h(V)$ is the same as the tensor algebra:
$$T^+_h(V):=\bigoplus_{n\geq 1} V^{\ot n}.$$
To equip it with a Hom-associative algebra structure, we first extend the linear map $\alpha_V$ on $V$ to a linear map $\az$ on $V^{\ot n}, n\geq 1,$ by the tensor multiplicativity: define
\begin{equation}
\az(\fraka)=\az(a_1\ot \cdots\ot a_n):=\alpha_V(a_1)\ot \cdots
\ot \alpha_V(a_n)
\mlabel{eq:defalpha0}
\end{equation}
for any pure tensor $\fraka:=a_1\ot \cdots \ot a_n\in V^{\ot n}$ and further extend to $T^+_h(V)$ by additivity. Note that $\az$ is compatible with the tensor product in $T^+_h(V)$, that is,
\begin{equation}
\az(\fraka\ot \frakb)=\az(\fraka)\ot \az(\frakb) \quad\text{for all}\, \fraka\in V^{\ot m}, \frakb\in V^{\ot n}.
\mlabel{eq:alpcom}
\end{equation}
Since $(V,\alpha_V)$ is involutive, we obtain
\begin{equation}
\az^2=\id.
\mlabel{eq:invol}
\end{equation}

We next define a binary operation $\ohz=\ohz_V$ on $T^+_h(V)$. For any pure tensors $\fraka=a_1\ot \cdots\ot a_m\in V^{\ot m}$ and $\frakb=b_1\ot \cdots \ot b_n\in V^{\ot n}$, we define $\fraka \ohz \frakb$ by induction on $n\geq 1$. When $n=1$, we have $\frakb\in V$ and then define
\begin{equation}
\fraka \ohz\frakb:=\fraka\ot \frakb
\mlabel{eq:ohini}
\end{equation}
Assume that $\fraka \ohz\frakb$ has been defined when $n\leq k$ with $k\geq 1$. Consider $\fraka=a_1\ot \cdots \ot a_m\in V^{\ot m}$ and $\frakb=b_1\ot \cdots b_{k+1}\in V^{\ot (k+1)}$. We then define
\begin{equation}
\fraka\ohz \frakb:=
\big(\az(\fraka)\ohz (b_1\ot \cdots \ot b_k)\big)\ot \az(b_{k+1})
\mlabel{eq:ohrec}
\end{equation}
and apply the induction hypothesis to the right hand side. There is also an explicit formula for $\ohz$:
\begin{lemma}
For $\fraka\in V^{\ot m}$ and $\frakb\in V^{\ot n}$ where $m, n\geq 1$, we have
\begin{equation}
\fraka\ohz \frakb=
\az^{n-1}(\fraka)\ot b_1\ot \az(b_2\ot \cdots \ot b_n).
\mlabel{eq:ohexp}
\end{equation}
\mlabel{lem:ohexp}
\end{lemma}
\begin{proof}
We prove by induction on $n\geq 1$ with the case when $n=1$ given in Eq.~(\mref{eq:ohini}). Assume that the lemma has been proved when $n=k$ for $k\geq 1$. For $\frakb\in V^{\ot (k+1)}$ we write $\frakb=\frakb^\prime\ot b_{k+1}$ with $\frakb^\prime:=b_1\ot \cdots \ot b_k$. Then by Eq.~(\mref{eq:ohrec}) and the induction hypothesis we derive
$$ \fraka\ohz \frakb=
\big(\az(\fraka)\ohz (b_1\ot \cdots \ot b_k)\big)\ot \az(b_{k+1})
=\az^k(\fraka)\ot b_1\ot \az(b_2\ot \cdots \ot b_k)
\ot \az(b_{k+1}),$$
as needed.
\end{proof}
Extending $\ohz$ biadditively, we obtain a binary operation
$$\ohz: T^+_h(V)\ot T^+_h(V)\rightarrow T^+_h(V).$$
Finally define
$$i_V:V\rightarrow T^+_h(V)$$
to be the canonical inclusion map.

The following is our main result on the free involutive Hom-associative algebra on an involutive Hom-module, generalizing~\cite{ZG}. It will be proved in the next subsection.
\begin{theorem}
Let $(V,\alpha_V)$ be an involutive Hom-module. Let $T^+_h(V), \,\az,\, \ohz$ and $i_V$ be defined as above.
Then
\begin{thmenumerate}
\item
The triple $T^+_h(V):=(T^+_h(V),\,\ohz,\,\az)$ is an involutive Hom-associative algebra.
\mlabel{it:inv0}
\item
The quadruple $(T^+_h(V),\,\ohz,\, \az,\, i_V)$ is the free involutive Hom-associative algebra on $V$.
\mlabel{it:freeinv0}
\end{thmenumerate}
\mlabel{thm:profree}
\end{theorem}

\subsection{The proof of Theorem~\mref{thm:profree}}
\mlabel{ss:freepf}
In this section we give the proof of Theorem~\mref{thm:profree}.
\smallskip

\noindent
(\mref{it:inv0}) The involution property of $\az$ is already established in Eq.~(\mref{eq:invol}). We first prove the multiplicativity
\begin{equation}
\az(u\ohz v)=\az(u)\ohz\az(v) \quad \text{for all } u,v\in T^+_h(V),
\mlabel{eq:ohcom}
\end{equation}
for which we just need to verify it for pure tensors $u=\fraka\in V^{\ot m}$ and  $v=\frakb=\frakb'\ot b_n\in V^{\ot n}$, where $m, n\geq 1$.
By Eqs.~(\mref{eq:invol}) and (\mref{eq:ohexp}) we have
$$\az(\fraka\ohz \frakb)=\az(\az^{n-1}(\fraka)\ot b_1\ot \az(\frakb'))
=\az^n(\fraka)\ot \az(b_1)\ot \frakb'$$
and
$$\az(\fraka)\ohz\az(\frakb)=\az^{n-1}(\az(\fraka))\ot \az(b_1)\ot \frakb'$$
as needed.

We next verify the Hom-associativity:
\begin{equation}
\az(\fraka)\ohz(\frakb\ohz \frakc)=(\fraka\ohz \frakb)\ohz \az(\frakc)\quad\text{for all } \fraka\in V^{\ot m},\frakb\in V^{\ot n}, \frakc\in V^{\ot\ell}.
\mlabel{eq:homass1}
\end{equation}
Again we only need to verify it when $\fraka\in \frakg^{\ot m}, \frakb=b_1\ot \frakb'\in \frakg^{\ot n}$ and $\frakc=c_1\ot \frakc'\in\frakg^{\ot \ell}$ are pure tensors. Then we have
$$\az(\fraka)\ohz(\frakb\ohz\frakc) =\az(\fraka)\ohz\big(\az^{\ell-1}(\frakb) \ot c_1\ot \az(\frakc')\big)
=\az^{n+\ell}(\fraka)\ot \az^{\ell-1}(b_1)\ot \az^\ell(\frakb')\ot \az(c_1)\ot \frakc'$$
$$(\fraka \ohz\frakb)\ohz \az(\frakc)
=(\az^{n-1}(\fraka)\ot b_1\ot \az(\frakb'))\ohz \az(\frakc)
=\az^{n-1+\ell-1}(\fraka)\ot\az^{\ell-1}(b_1)\ot \az^\ell(\frakb')\ot \az(c_1)\ot \az^2(\frakc'),$$
again giving what is needed by Eq.~(\mref{eq:invol}).
\smallskip

\noindent
(\mref{it:freeinv0}) Let $(A, \ast, \alpha_A)$ be an involutive Hom-associative algebra. Let $f: V\rightarrow A$ be a  morphism of Hom-modules. We will construct a morphism of Hom-associative algebras
$$\bar{f}: T^+_h(V)\rightarrow A$$
by inductively defining $\bar{f}(\fraka)$ for pure tensors $\fraka\in V^{\ot n}$. For $n=1$, we have $\fraka\in V$ and define
\begin{equation}
\bar{f}(\fraka):=f(\fraka).
\mlabel{eq:fbar1}
\end{equation}
Inductively, for any pure tensor $\fraka=a_1\ot a_2\ot \cdots \ot a_{n+1}=\fraka'\ot a_{n+1}\in V^{\ot (n+1)}$ with $\fraka'=a_1\ot\cdots\ot a_n\in V^{\ot n}$, we define
\begin{equation}
\bar{f}(\fraka):=\bar{f}(\fraka')\ast\bar{f}(a_{n+1}).
\mlabel{eq:deff}
\end{equation}
Since $i_V(a)=a$ for $a\in V$, we have
\begin{equation}(\bar{f}\circ i_V)(a)=f(a).
\end{equation}

We next prove that the $\bfk$-linear map $\bar{f}$ defined above is indeed a morphism of Hom-associative algebras. We first verify that $\bar{f}$ satisfies
\begin{equation}
\bar{f}(\az(\fraka))=\alpha_A(\bar{f}(\fraka)) \quad \text{ for all } \fraka\in T^+_h(V),
\mlabel{eq:fgam}
\end{equation}
checking it for any pure tensor $\fraka\in V^{\ot n}$ by induction on $n\geq 1$. When $n=1$, we have $\fraka\in V$.  Since $f:(V,\av)\rightarrow (A,\alpha_A)$ is a morphism of Hom-modules, by the definition of $\bar{f}$, we have
\begin{equation}
\bar{f}(\az(\fraka))=f(\av(\fraka))
=\alpha_A(f(\fraka))=\alpha_A(\bar{f}
(\fraka)).
\end{equation}
For the inductive step, consider $\fraka=\fraka'\ot a_{n+1}\in V^{\ot (n+1)}$ with $\fraka'\in V^{\ot n}, n\geq 1$. Then
\begin{eqnarray*}
\bar{f}(\az(\fraka))&=&\bar{f}(\az(\fraka')\ot \az(a_{n+1}))\quad\text{(by Eq.~(\mref{eq:alpcom}))}\\
&=&\bar{f}(\az(\fraka'))\ast\bar{f}
({\az(a_{n+1})})\quad\text{(by Eq.~(\mref{eq:deff}))}\\
&=&\alpha_A(\bar{f}(\fraka'))
\ast\alpha_A(\bar{f}(a_{n+1}))\quad
\text{(by the induction hypothesis)}\\
&=&\alpha_A(\bar{f}(\fraka')\ast\bar{f}(a_{n+1}))\\
&=&\alpha_A(\bar{f}(\fraka))
\quad\text{(by Eq.~(\mref{eq:deff}))},
\end{eqnarray*}
completing the inductive proof of Eq.~(\mref{eq:fgam}).

We next verify
\begin{equation}
\bar{f}(\fraka\ohz \frakb)=\bar{f}(\fraka)\ast \bar{f}(\frakb) \quad \text{for all } \fraka,\frakb\in T^+_h(V),
\mlabel{eq:homf}
\end{equation}
by checking it for  pure tensors $\fraka\in V^{\ot m}$ and $\frakb\in V^{\ot m}$ inductively on $n\geq 1$. When $n=1$, we have $\frakb\in V$. By the definition of $\bar{f}$, we get
$$\bar{f}(\fraka\ohz \frakb)=\bar{f}(\fraka\ot \frakb)=\bar{f}(\fraka)\ast \bar{f}(\frakb).$$
Assuming that Eq.~(\mref{eq:homf}) has been proved for $m\leq k$ for a $k\geq 1$, consider $\frakb=\frakb'\ot b_{k+1}\in V^{\ot (k+1)}$, where $\frakb'\in V^{\ot k}$. Then we obtain
\begin{eqnarray*}
\bar{f}(\fraka\ohz \frakb)&=& \bar{f}((\az(\fraka)\ohz \frakb')\ot \az(b_{k+1}))\quad\text{(by Eq.~(\mref{eq:ohrec}))}\\
&=&\bar{f}((\az(\fraka)\ohz \frakb'))\ast \bar{f}(\az(b_{k+1}))\quad\text{(by Eq.~(\mref{eq:deff}))}\\
&=&\big(\bar{f}(\az(\fraka))\ast \bar{f}(\frakb')\big)\ast \alpha_A(\bar{f}(b_{k+1}))\quad\text{(by Eq.~(\mref{eq:fgam}) and the induction hypothesis)}\\
&=&\alpha_A(\bar{f}(\az(\fraka)))\ast \big(\bar{f}(\frakb')\ast \alpha_A^2(\bar{f}(b_{k+1}))\big)\quad\text{(by Hom-associativity)}\\
&=&\bar{f}(\fraka)\ast \big(\bar{f}(\frakb')\ast \bar{f}(b_{k+1})\big)\quad\text{(by $\alpha_A^2=\id$ and Eq.~(\mref{eq:fgam}))}\\
&=&\bar{f}(\fraka)\ast \bar{f}(\frakb)
\quad\text{(by Eq.~(\mref{eq:deff}))}.
\end{eqnarray*}
This completes the induction.

We finally prove the uniqueness of $\bar{f}$. Suppose that there is another
morphism of Hom-associative algebras $\tilde{f}:T^+_h(V)\rightarrow A$ such that $\tilde{f}\circ i_V=f$.
We prove $\bar{f}(\fraka)=\tilde{f}(\fraka)$ for $\fraka\in V^{\ot n}$ by induction on $n\geq 1$. When $n=1$, we have $\fraka\in V$.
Then we get
$$\tilde{f}(\fraka)=(\tilde{f}\circ i_V)(\fraka)=f(\fraka)=\bar{f}(\fraka).$$
Assume $\bar{f}(\fraka)=\tilde{f}(\fraka)$ holds for $\fraka\in V^{\ot n}$ with $n\geq 1$.
Let $\fraka=\fraka'\ot a_{n+1}=\fraka'\ohz a_{n+1}\in V^{\ot (n+1)}$, where $\fraka'\in V^{\ot n}$. Then applying the induction hypothesis, we obtain
$$\bar{f}(\fraka)=\bar{f}(\fraka'\ohz a_{n+1})=\bar{f}(\fraka')\ast\bar{f}(a_{n+1})
= \tilde{f}(\fraka')\ast\tilde{f}(a_{n+1})
=\tilde{f}(\fraka' \ohz a_{n+1}) =\tilde{f}(\fraka).$$
Then we conclude $\bar{f}=\tilde{f}$.

This completes the proof of Theorem~\mref{thm:profree}.(\mref{it:freeinv0}) and therefore the proof of Theorem~(\mref{thm:profree}).

\section{Universal enveloping Hom-associative algebra of an involutive Hom-Lie algebra}
\mlabel{sec:env}

The first construction of the enveloping Hom-associative algebra of a Hom-Lie algebra was obtained by D. Yau in~\cite[Theorem~2]{Ya4}. He obtained his construction as a suitable quotient of his explicit construction by trees, of the free Hom-nonassociative algebra on the Hom-module underlying the given Hom-Lie algebra. Applying our explicit construction of free involutive Hom-associative algebras in the previous section, we will provide a construction of the enveloping Hom-associative algebra of an involutive Hom-Lie algebra. Its tensor form makes it more convenient for the study of the Poincar\'e-Birkhoff-Witt type theorem in Section~\mref{sec:pbw}.

We begin with recalling the definition of a Hom-Lie algebra~\cite{HLS}.
\begin{defn}
\begin{enumerate}
\item
A {\bf Hom-Lie algebra} is a triple $(\frakg,\brag,\beta_\frakg)$ consisting of a $\bfk$-module $\frakg$, a  bilinear map $\brag:\frakg\times \frakg\rightarrow \frakg$ and a linear map $\beta_\frakg:\frakg\rightarrow\frakg$ satisfying
\begin{equation}
[x,y]_\frakg=-[y,x]_\frakg,
\mlabel{eq:skewsym}
\end{equation}
\begin{equation}
(\text{\bf Hom-Jacobi identity})\quad [\beta_\frakg(x),[y,z]_\frakg]_\frakg
+[\beta_\frakg(y),[z,x]_\frakg]_\frakg
+[\beta_\frakg(z),[x,y]_\frakg]_\frakg=0,
\mlabel{eq:homjacobi}
\end{equation}
\begin{equation}
\beta_\frakg([x,y]_\frakg)=[\beta_\frakg(x),\beta_\frakg(y)]_\frakg
\quad\text{for all } x,y,z\in\frakg.
\mlabel{eq:hommult}
\end{equation}
\item
A {\bf morphism of Hom-Lie algebras} $f:(\frakg,\brag,\beta_\frakg)
\rightarrow(\frakk,[,]_\frakk, \beta_\frakk)$ is a $\bfk$-linear map
$f:\frakg\rightarrow\frakk$ such that
$$
f([x,y]_\frakg)=[f(x),f(y)]_\frakk \quad
\text{and}\quad
f(\beta_\frakg(x)=\beta_\frakk(f(x))
\quad \text{for all } x, y\in\frakg.
$$
\item
A Hom-Lie algebra $(\frakg,\brag,\beta_\frakg)$ is called {\bf involutive} if $\beta_\frakg^2=\id$.
\end{enumerate}
\end{defn}

As in the case of an associative algebra and a Lie algebra, a Hom-associative algebra $(A,\cdot,\alpha_A)$ gives a Hom-Lie algebra by antisymmetrization. We denote this Hom-Lie algebra by $(A,[\,,\,]_A,\beta_A)$. So $\beta_A=\alpha_A$ and $[x,y]_A=xy-yx$ for $x, y\in A$.

\begin{defn}Let $\frakg:=(\frakg,\brag,\beta_\frakg)$ be a Hom-Lie algebra.
A {\bf universal enveloping Hom-associative algebra} of $\frakg$ is a Hom-associative algebra $\frakU_\frakg:=(\frakU_\frakg,\ast_\frakU,\alpha_\frakU)$, together with a morphism
$\phi_\frakg:(\frakg, [\,,\,]_\frakg,\beta_\frakg)\rightarrow (\frakU_\frakg,[\,,\,]_{\frakU_\frakg},\beta_{\frakU_\frakg})$ of Hom-Lie algebras, that satisfies the following universal property: for any Hom-associative algebra $A:=(A,\cdot_A, \alpha_A)$, and any Hom-Lie algebra morphism $\xi:(\frakg, [\,,\,]_\frakg, \beta_\frakg)\rightarrow (A,[\,,\,]_A,\beta_A)$, there exists a unique morphism $\hat{\xi}: \frakU_\frakg\rightarrow A$ of Hom-associative algebras such that $\hat{\xi} \phi_\frakg=\xi$.
\mlabel{de:ueh}
\end{defn}

We next show that in the involutive case, the verification of the universal property can be suitably weakened.

\begin{lemma}
Let $(\frakg,\brag, \beta_\frakg)$ be an involutive Hom-Lie algebra.
\begin{enumerate}
\item
Let $(A,\cdot,\alpha_A)$ be a Hom-associative algebra. Let $f:(\frakg,\brag,\beta_\frakg) \to (A,[\,,\,]_A,\beta_A)$ be a morphism of Hom-Lie algebras and let $B$ be the Hom-associative subalgebra of $A$ generated by $f(\frakg)$. Then $B$ is involutive.
\mlabel{it:invenv1}
\item
The universal enveloping Hom-associative algebra $(\frakU_\frakg,\phi_\frakg)$ of  $(\frakg,\brag, \beta_\frakg)$ is involutive.
\mlabel{it:invenv2}
\item
In order to verify the universal property of $(\frakU_\frakg,\phi_\frakg)$ in Definition~\mref{de:ueh}, we only need to consider involutive Hom-associative algebras $A:=(A,\cdot_A, \alpha_A)$.
\mlabel{it:invenv3}
\end{enumerate}
\mlabel{lem:invenv}
\end{lemma}

\begin{proof}
(\mref{it:invenv1})
Let
\begin{equation}
 C:=\{x\in A\,|\, \alpha_A^2(x)=x\}.
 \mlabel{eq:subinv}
\end{equation}
Then $C$ is clearly a submodule. Also for $x, y\in C$, from $\alpha_A(xy)=\alpha_A(x)\alpha_A(y)$, we have $\alpha_A^2(xy)=\alpha_A^2(x)\alpha_A^2(y)=xy$ and hence $xy\in C$. Further, from $x\in C$, we obtain $\alpha_A^2(\alpha_A(x))=\alpha_A(\alpha_A^2(x))=\alpha_A(x)$ and hence $\alpha_A(x)\in C$. Therefore $C$ is a Hom-associative subalgebra of $A$. 
Since $f$ is a morphism of Hom-Lie algebras, from the equation $\beta_\frakg^2=\id$, we obtain
\begin{equation}
\alpha_A^2(f(x))=\beta_A^2(f(x))=f(\beta_\frakg^2
(x))=f(x).
\mlabel{eq:betaA2}
\end{equation}
Hence $\im f\subseteq C$. Therefore $C$ is a Hom-associative subalgebra of $A$ that contains $\im f$. Then $C$ contains $B$ and so $B$ is involutive.
\smallskip

\noindent
(\mref{it:invenv2}) is a special case of Item~(\mref{it:invenv1}) since, from its universal property, $\frakU_\frakg$ is generated by $\phi_\frakg(\frakg)$ as a Hom-associative algebra.
\smallskip

\noindent
(\mref{it:invenv3})
We only need to prove that, assuming that the universal property of $\frakU_\frakg$ holds for involutive Hom-associative algebras, then it holds for all Hom-associative algebras. Thus let $(A,\cdot,\alpha_A)$ be a Hom-associative algebra and let $\xi:(\frakg,\brag,\beta_\frakg)\to (A,[\,,\,]_A,\beta_A)$ be a morphism of Hom-Lie algebras. Let $C=\{x\in A\,|\,\alpha_A^2(x)=x\}$ be the involutive Hom-associative subalgebra of $A$ defined in Eq.~(\mref{eq:subinv}). By Item~(\mref{it:invenv1}) and its proof, $\im\, \xi$ is contained in $C$ and thus $\xi$ is the composition of a morphism $\xi_C:(\frakg,\brag,\beta_\frakg)\to (C,[\,,\,]_C,\beta_C)$ of Hom-associative algebras with the inclusion $C\to A$. By our assumption, there is a morphism $\hat{\xi}_C:\frakU_\frakg\to C$ of Hom-associative algebras such that $\hat{\xi}_C\phi_\frakg=\xi_C$. Then composing with the inclusion $B\to A$, we obtain a morphism $\hat{\xi}:\frakU_\frakg\to A$ of Hom-associative algebras such that $\hat{\xi}\phi_\frakg=\xi$. Further, suppose $\hat{\xi}^\prime:\frakU_\frakg \to A$ is another morphism of Hom-associative algebras such that $\hat{\xi}^\prime\phi_\frakg=\xi$. Since $\frakU_\frakg$ is involutive by Item~(\mref{it:invenv2}), $\im \, \hat{\xi}^\prime$ is involutive. Hence $\hat{\xi}^\prime$ is the composition of a morphism $\hat{\xi}^\prime_C:\frakU_\frakg\to C$ with the inclusion $C\to A$. Also, $\hat{\xi}^\prime_C \phi_\frakg=\xi_C$. Since $C$ is involutive, by our assumption at the beginning of the proof, the morphisms $\hat{\xi}_C$ and $\hat{\xi}^\prime_C$ coincide. Then $\hat{\xi}$ and $\hat{\xi}^\prime$ coincide. This completes the proof.
\end{proof}

Now we are ready to give our construction of the universal enveloping Hom-associative algebra of an involutive Hom-Lie algebra.
\begin{theorem}
Let $\frakg:=(\frakg,\brag,\beta_\frakg)$ be an involutive Hom-Lie algebra. Let $T^+_h(\frakg):=(T^+_h(\frakg),\ohz,\az)$ be the free Hom-associative algebra on the Hom-module underlying $\frakg$ obtained in Theorem~\mref{thm:profree}. Let $I_{\frakg,\beta}$ be the Hom-ideal of $T^+_h(\frakg)$ generated by the set
\begin{equation}
\left \{x\ot y-y\ot x-[x,y]_\frakg\,|\, x, y\in \frakg\right\}
\mlabel{eq:comhom}
\end{equation}
and let
$$\ug:=T^+_h(\frakg)/I_{\frakg,\beta}$$
be the quotient Hom-associative algebra. Let
$$\phi_0:=\pi\circ i_\frakg:\frakg\,{\overset{i_\frakg}{\longrightarrow}}\, T^+_h(\frakg) \,{\overset{\pi}{\longrightarrow}}\,\ug, \quad x\mapsto x+I_{\frakg,\beta},$$
be the composition of the natural inclusion $i_\frakg: \frakg\to T^+_h(\frakg)$ with the quotient map $\pi:T^+_h(\frakg)\to \ug$. Then
$(\ug,\phi_0)$
is a universal enveloping Hom-associative algebra of $\frakg$. The universal enveloping Hom-associative algebra of $\frakg$ is unique up to isomorphism.
\mlabel{thm:isougd}
\end{theorem}

\begin{proof}
Let $\oh$ denote the multiplication on $\ug$.

We first verify that $\phi_0:(\frakg, [\,,\,]_\frakg,\beta_\frakg)\rightarrow (\frakU_\frakg,[\,,\,]_{\frakU_\frakg},\beta_{\frakU_\frakg})$ is a morphism of Hom-Lie algebras.
Since both $i_\frakg$ and $\pi$ are Hom-module morphisms, so is their composition $\phi_0$. Since $x\ot y-y\ot x-[x,y]_\frakg$ is in $I_{\frakg,\beta}=\ker \pi$ for $x,y\in\frakg$, we obtain
\begin{align*}
\phi_0([x,y]_\frakg)
&=\pi([x,y]_\frakg)
=\pi(x\ot y-y\ot x)
=\pi(x\ohz y-y\ohz x)
=\pi(x)\oh\pi(y)-\pi(y)\oh\pi(x)\\
&=\phi_0(x)\oh\phi_0(y)
-\phi_0(y)\oh\phi_0(x)
= [\phi_0(x),\phi_0(y)]_{\ug},
\end{align*}
as needed.

We next prove that $\ug$ satisfies the desired universal property. By Lemma~\mref{lem:invenv}.(\mref{it:invenv3}). We only need to consider  an arbitrary Hom-associative algebra $A:=(A,\ast_A,\alpha_A)$ that is involutive.
Let $\xi:(\frakg, [\,,\,]_\frakg,\beta_\frakg)\rightarrow (A,[\,,\,]_A,\beta_A)$ be a morphism of Hom-Lie algebras.
Since $T^+_h(\frakg)$ is the free involutive Hom-associative algebra on the underlying Hom-module of $\frakg$, by Theorem~\mref{thm:profree}, there exists a Hom-associative algebra morphism $\bar{\xi}:T^+_h(\frakg)\rightarrow A$ such that $\bar{\xi}\circ i_\frakg=\xi$. Since $\xi$ is a morphism of Hom-Lie algebras, we have
\begin{align*}
&\bar{\xi}(x\ot y-y\ot x)=\bar{\xi}(x\ohz y-y\ohz x)
=\bar{\xi}(x)\ast_A\bar{\xi}(y)-\bar{\xi}(y)\ast_A\bar{\xi}(x)\\
=&\xi(x)\ast_A\xi(y)-\xi(y)\ast_A\xi(x)
=[\xi(x),\xi(y)]_A = \xi([x,y]) = \bar{\xi}([x,y]).
\end{align*}
Thus $I_{\frakg,\beta}$ is contained in $\ker\,\bar{\xi}$ and $\bar{\xi}$ induces a morphism $\hat{\xi}:\ug\rightarrow A$ of Hom-associative algebras such that
$\bar{\xi} =\hat{\xi} \pi$. Then $\hat{\xi} \phi_0=\hat{\xi} \pi i_\frakg=\bar{\xi}i_\frakg = \xi.$

It remains to prove the uniqueness of $\hat{\xi}$.  Suppose that there exists another morphism $\hat{\xi}^\prime:\ug\rightarrow A$ of Hom-associative algebras such that $\hat{\xi}^\prime \phi_0=\xi$.
We just need to prove $\hat{\xi}(u)=\hat{\xi}^\prime(u)$ for any $u\in \ug$. Since $T^+_h(\frakg)=\bigoplus_{n\geq 1} \frakg^{\ot n}$,
we just need to prove
\begin{equation}
\hat{\xi}\pi(\fraka)=\hat{\xi}^\prime\pi(\fraka)
\mlabel{eq:uniq}
\end{equation}
for $\fraka\in \frakg^{\ot n}$ with $n\geq 1$, for which we apply the induction on $n\geq 1$. For $n=1$, namely for $\fraka\in \frakg$, we have
$$\hat{\xi}\pi(\fraka)=\hat{\xi}\pi i_\frakg(\fraka)
=\bar{\xi}i_\frakg(\fraka)
=\xi(\fraka)
=\hat{\xi}^\prime\phi_0(\fraka)=
\hat{\xi}^\prime\pi(\fraka).
$$
Assume that Eq.~(\mref{eq:uniq}) has been proved for $n\geq 1$. Let $\fraka=\fraka'\ot a_{n+1}\in \frakg^{\ot (n+1)}$, where $\fraka'\in \frakg^{\ot n}$. Since $\hat{\xi}\pi$ and $\hat{\xi}^\prime \pi$ are morphisms of Hom-associative algebras, by Eq.~(\mref{eq:ohini}), the induction hypothesis and the initial step, we have
$$
\hat{\xi}\pi(\fraka)=\hat{\xi}\pi(\fraka'\ohz a_{n+1})
=\hat{\xi}(\pi(\fraka'))\ast_A \hat{\xi} (\pi(a_{n+1}))
=\hat{\xi}^\prime(\pi(\fraka'))\ast_A \hat{\xi}^\prime (\pi(a_{n+1}))
=\hat{\xi}^\prime\pi(\fraka'\ohz a_{n+1})=\hat{\xi}^\prime(\pi(\fraka)).
$$
This completes the inductive proof of the uniqueness of $\hat{\xi}$. Thus $(\ug,\phi_0)$ is a universal enveloping algebra of $\frakg$.

The proof of the uniqueness of $\ug$ up to isomorphism is standard. We include the details for completeness. Suppose that $(\ug_1,\phi_1)$ is another universal enveloping algebra of $\frakg$. By the universal property of $(\ug,\phi_0)$ and $(\ug_1,\phi_1)$, there exist homomorphisms $f:\ug\rightarrow\ug_1$ and
$f_1:\ug_1\rightarrow \ug$ of Hom-associative algebras such that $f  \phi_0=\phi_1$ and $f_1  \phi_1=\phi_0$.
Then
$$(f_1  f)  \phi_0=\phi_0=\id_{\ug}  \phi_0.$$
Since $\id_{\ug}$ and $f_1  f$ are both Hom-associative homomorphisms, by the uniqueness in the universal property of $(\ug,\phi_0)$, we have $f_1  f=\id_{\ug}$. Similarly, $f  f_1=\id_{\ug_1}$. This shows that $f $ and $f_1$ are isomorphisms. Thus $\ug$ is unique up to isomorphism.
\end{proof}

\section{The Poincar\'e-Birkhoff-Witt theorem for involutive Hom-Lie algebras}
\mlabel{sec:pbw}

In this section we prove a Poincar\'e-Birkhoff-Witt (PBW) type theorem, namely Theorem~\mref{thm:ncpbw}, for involutive Hom-Lie algebras.
After giving some motivation and two preparatory results, the statement of the theorem together and its proof modulo the two preparatory results are given in Section~\mref{ss:hompbw}. Proofs of the two preparatory results are provided in Sections~\mref{ss:linsys} and \mref{ss:parpbw} respectively.

\subsection{The Poincar\'e-Birkhoff-Witt theorem}
\mlabel{ss:hompbw}

Let $\frakg$ be a Lie algebra with an ordered basis $X=\{x_i\,|\,i\in \omega \}$ indexed by a well ordered set $\omega$. Let $I_\frakg$ be the ideal of the free associative algebra $T^+(\frakg)$ on $\frakg$ generated by the set
\begin{equation}
 \left\{x\ot y-y\ot x -[x,y]_\frakg \,|\, x, y \in \frakg\right\},
 \mlabel{eq:comlie}
\end{equation}
so that $U_\frakg:=T^+(\frakg)/I_\frakg$ is the universal enveloping algebra of $\frakg$. The Poincar\'e-Birkhoff-Witt (PBW) Theorem states that the linear subspace $I_\frakg$ of $T^+(\frakg)$ has a canonical linear complement which has a basis given by
\begin{equation}
 W:=\{x_{i_1}\ot \cdots \ot x_{i_n}\,|\, i_1\geq \cdots \geq i_n, n\geq 1\},
 \mlabel{eq:pbwbasis}
\end{equation}
called the PBW basis of $U_\frakg$. One of its proofs, as presented in~\cite{Var}, is based on the fact that $I_\frakg$ is linearly generated by the set
\begin{equation}
\left\{\fraka \ot (x\ot y-y\ot x-[x,y]_\frakg)\ot \frakb \,|\, x, y\in X, \fraka\in \frakg^{\ot m}, \frakb\in \frakg^{\ot n}, m, n\geq 0\right\}.
\mlabel{eq:liegen}
\end{equation}
Then the proof is essentially to show that the rewriting system~\cite{BN} from the above linear generators is convergent.

Now let $(\frakg,\beta_\frakg)$ be an involutive Hom-Lie algebra. The tensor algebra like construction $T^+_h(\frakg)$ of the free involutive Hom-associative algebra on $\frakg$ obtained in Theorem~\mref{thm:profree} and the resulting universal enveloping Hom-associative algebra $\ug:=T^+_h(\frakg)/I_{\frakg,\beta}$ in Theorem~\mref{thm:isougd} suggest that the proof of the classical PBW theorem might be adapted to prove a PBW theorem for involutive Hom-Lie algebras. This is the motivation for our approach. But there are obstacles to overcome.

The first obstacle in this approach is the complexicity of the multiplication $T^+_h(\frakg)$. The multiplication in $T^+(\frakg)$ is simply the tensor concatenation product, while the one in $T^+_h(\frakg)$ is not. Further the generating of a Hom-ideal is more complicated then the generating of an ideal for the lack of the associativity.

Thus our first task is express a linear generating system of $I_{\frakg,\beta}$ in terms of the tensor product. We carry out this task in Section~\mref{ss:linsys}, in two steps. We first unravel the iterated multiplications in $I_{\frakg,\beta}$ to show that only two iterations of multiplication is necessary. We then show that the resulting generators, after the application of a suitable twist operator $\varphi$, has a tensor form that is similar to the linear generators in Eq.~(\mref{eq:liegen}) in the Lie algebra case.
This is presented in our first preparatory result, Proposition~\mref{prop:ncideal}, for the PBW theorem of involutive Hom-Lie algebras.

To simplify notations, we denote $\bar{a}:=\beta_\frakg(a)$ for $a\in \frakg$ and more generally $\bar{\fraka}:=\az(\fraka)$ for $\fraka\in \frakg^{\ot n}, n\geq 1$.
We next define a linear operator
\begin{equation}
\varphi:\frakg^{\ot n}\longrightarrow \frakg^{\ot n}, \quad \fraka:=a_1\ot a_2\ot \cdots \ot a_n\mapsto \tilde{a}_1\ot \tilde{a}_2\ot \cdots \ot \tilde{a}_n,
\mlabel{eq:defphi}
\end{equation}
where
$$
\tilde{a}_i=\left\{\begin{array}{ccc}
\bar{a_i},& \quad \quad i=2k+1~ \text{and}~ k\geq 1;\\
a_i,&\text{otherwise}.
\end{array}\right.
$$
In other words, $\varphi$ involutes the 3rd, 5th, 7th, $\cdots$, tensor factor of a pure tensor $\fraka$.

\begin{prop}
Let $(\frakg,[\,,\,]_\frakg,\beta_\frakg)$ be an involutive Hom-Lie algebra. Let $\varphi:T^+_h(\frakg)\to T^+_h(\frakg)$ be as above. Let $I_{\frakg,\beta}$ be the Hom-ideal of $T^+_h(\frakg)$ defined in Theorem~\mref{thm:isougd} such that $\ug=T^+_h(\frakg)/I_{\frakg,\beta}$ is the universal enveloping algebra of $\frakg$. Denote
\begin{equation}
J_{\frakg,\beta}:=
\sum_{i,j\geq0}\sum_{\substack{\fraka\in \frakg^{\ot i}\\\frakb\in \frakg^{\ot j}}} \sum_{x,y\in \frakg}  \bigg( \fraka\ot(x\ot y-y\ot x )\ot \frakb - \az(\fraka)\ot [x,y]_\frakg \ot \az(\frakb)\bigg).
\mlabel{eq:Jgb}
\end{equation}
Then
$$
\varphi(I_{\frakg,\beta})=J_{\frakg,\beta}
$$
\mlabel{prop:ncideal}
\end{prop}
This proposition will be rephrased as Proposition~\mref{prop:ncideal1} and proved in Section~\mref{ss:linsys}.

Our second challenge in adapting the rewriting system proof of the classical PBW Theorem is that there might not be a basis of $\frakg$ that is stable under the action of the Hom map $\beta_\frakg$, but only does so up to a sign. In order to give a uniform approach to include the possible cases, we first prove the following PBW decomposition which includes a parameter, applicable to the various cases of $\beta_\frakg$ by taking various values of the parameter, eventually leading to the proof of the Hom-PBW Theorem~\mref{thm:ncpbw}.

\begin{theorem} $(\text{\rm PBW Decomposition with a Parameter})$
Let $(\frakg,[\,,\,]_\frakg,\beta_\frakg)$ be an involutive Hom-Lie algebra such that $\beta_\frakg(X)=X$ for a well-ordered basis $X$ of $\frakg$. Let $W$ be as defined in Eq.~(\mref{eq:pbwbasis}). Let $\mu\in \bfk$ be given.
Denote
\begin{equation}
J_{\frakg,\beta,\mu}:=
\sum_{i,j\geq0}\sum_{\substack{\fraka\in \frakg^{\ot i}\\\frakb\in \frakg^{\ot j}}} \sum_{x,y\in \frakg}  \bigg( \fraka\ot(x\ot y-y\ot x )\ot \frakb - \mu^{i+j}\az(\fraka)\ot [x,y]_\frakg \ot \az(\frakb)\bigg)
\mlabel{eq:jgbeta}
\end{equation}
Then we have the decomposition
\begin{equation}
T^+_h(\frakg)=J_{\frakg,\beta,\mu}\oplus \bfk W.
\end{equation}
\mlabel{thm:genpbw}
\end{theorem}
This theorem will restated as Theorem~\mref{thm:genpbw1} and proved in Section~\mref{ss:parpbw}.

With these two preparatory results at our disposal, we can now state and prove our Poincar\'e-Birkhoff-Witt theorem for involutive Hom-Lie algebras.

\begin{theorem}
Let $\bfk$ be a field whose characteristic is not 2. $(\frakg,[\,,\,]_\frakg,\beta_\frakg)$ be an involutive Hom-Lie algebra on $\bfk$. Let $\varphi:T^+_h(\frakg)\to T^+_h(\frakg)$ be as in Eq.(\mref{eq:defphi}). Let $I_{\frakg,\beta}$ be the Hom-ideal of $T^+_h(\frakg)$ generated by the commutators defined in Theorem~\mref{thm:isougd}. Let $J_{\frakg,\beta}$ be as defined in Eq.(\mref{eq:Jgb}). Then there is a well-ordered basis $X$ of $\frakg$ such that, for
$$ W:=W_X:=\{x_{i_1}\ot \cdots \ot x_{i_n}\,|\, i_1\geq \cdots \geq i_n, n\geq 1\},$$
the following statements hold.
\begin{enumerate}
\item
\begin{equation}
T^+_h(\frakg)=J_{\frakg,\beta}\oplus \bfk W.
\end{equation}
\mlabel{it:ncpbw2}
\item
{\bf (PBW Theorem)} $\varphi(W)$ is a basis of $T^+_h(\frakg)/I_{\frakg,\beta}$.
\mlabel{it:ncpbw3}
\end{enumerate}
\mlabel{thm:ncpbw}
\end{theorem}

\begin{proof}
(\mref{it:ncpbw2}) Let $\frakg_+$ and $\frakg_-$ be the eigenspace of $1$ and $-1$ of $\frakg$ respectively. Then from $x=\frac{\beta(x)+x}{2}+\frac{\beta(x)-x}{2}\in \frakg_++\frakg_-$ we have the decomposition
$$ \frakg=\frakg_+\oplus \frakg_-.$$
Let $Y_+$ and $Y_-$ be a basis of $\frakg_+$ and $\frakg_-$ respectively. Let $|Z|$ denote the cardinality of a set $Z$. We proceed with the proof in two cases:

\noindent
{\bf Case 1. ~\,$|Y_+|\geq |Y_-|$.~}
Fix an injection $\tau:Y_-\to Y_+$. Then the set
$$ X:=\{ \tau(x)+x, \tau(x)-x\,|\, x\in Y_-\}\cup (Y_+\backslash Y_-)$$ is a basis of $\frakg$ and $\beta_\frakg(X)=X$ since $\beta_\frakg(\tau(x)\pm x)=\tau(x)\mp x$. Let $W$ be defined with respect to a given well-order on $X$.
Take ~$\mu=1$ in Eq.~(\mref{eq:jgbeta}). Then by Eq.~(\mref{eq:Jgb}), $J_{\frakg,\beta,1}=J_{\frakg,\beta}$. By Theorem~\mref{thm:genpbw}, $T^+_h(\frakg)=J_{\frakg,\beta,1}\oplus \bfk W$. Thus $T^+_h(\frakg)=J_{\frakg,\beta}\oplus \bfk W$.
\smallskip

\noindent
{\bf Case 2.  ~\,$|X_+|\leq |X_-|$.~} Let $\gamma:=-\beta_\frakg$. Then it is direct to check that $(\frakg,[\,,\,]_\frakg,\gamma)$ is also an involutive Hom-Lie algebra. Further for this involutive Hom-Lie algebra, we are in Case 1. Take $\beta$ to be $\gamma$ and take ~$\mu=-1$ in Eq.~(\mref{eq:jgbeta}). Then by Eq.~(\mref{eq:Jgb}),  $J_{\frakg,\gamma,-1}=J_{\frakg,-\beta,-1}=J_{\frakg,\beta}$. Then by Theorem~\mref{thm:genpbw}, we obtain $$T^+_h(\frakg)=J_{\frakg,-\beta,-1}\oplus \bfk W=J_{\frakg, \beta}\oplus \bfk W.$$
\smallskip

\noindent
(\mref{it:ncpbw3}) follows from Proposition~\mref{prop:ncideal} and Item (\mref{it:ncpbw2}):
$$ T^+_h(\frakg)=\varphi(T^+_h(\frakg))=\varphi(J_{\frakg,\beta}\oplus \bfk W)=\varphi(J_{\frakg,\beta})\oplus \bfk \varphi(W)
=I_{\frakg,\beta}\oplus \bfk \varphi(W).$$
\end{proof}

\subsection{Tensor representation of the Hom-ideal $I_{\frakg,\beta}$}
\mlabel{ss:linsys}

In this subsection we study linear generators of the Hom-ideal $I_{\frakg,\beta}$ and express them in terms of the tensor product.

From $\beta_\frakg^2=\id$, $\beta_\frakg$ is bijective. So
\begin{equation}
\beta_\frakg(\frakg)=\frakg.\mlabel{eq:betag}
\end{equation}
It also gives $\varphi^2=\id$. So we also have
\begin{equation}
\varphi(\frakg^{\ot n})=\frakg^{\ot n}.\mlabel{eq:phifrakg}
\end{equation}

We next give some properties for the linear operator $\az$ and the multiplication $\ohz$.
First by Eqs.~(\mref{eq:defalpha0}) and~(\mref{eq:betag}), we obtain
\begin{equation}
\az^m(\frakg^{\ot n})=\frakg^{\ot n}, m\geq 0, n\geq 1. \mlabel{eq:alphbn}
\end{equation}
Then for any natural numbers $r, s\geq 1$, by Eq.(\mref{eq:ohexp}) we have
\begin{equation}
\frakg^{\ot r} \ohz \frakg^{\ot s}= \az^{s-1}(\frakg^{\ot r})\ot \frakg\ot \az(\frakg^{\ot (s-1)}) = \frakg^{\ot r+s}. \mlabel{eq:otsum}
\end{equation}

For $\fraka:=a_1\ot a_2\ot \cdots\ot a_n\in \frakg^{\ot n}$ and $n\geq 1$, we denote $\ell(\fraka):=n$.
\begin{lemma}
Let $\fraku:=u_1\ot u_2\ot \cdots\ot u_m\in \frakg^{\ot m}, \frakv:=v_1\ot v_2\ot \cdots\ot v_k\in \frakg^{\ot k}$ and $\frakw:=w_1\ot w_2\ot \cdots\ot w_n\in \frakg^{\ot m}$. Let $\varphi$ be defined as in Eq.~(\mref{eq:defphi}).
Then
\begin{enumerate}
\item
$\varphi(\fraku)=u_1\ot u_2\ot \az(u_3) \ot \az^2(u_4) \ot \cdots\ot \az^{m-2}(u_m)=u_1\ot \ot_{k=2}^m \az^{k-2}(u_k).
$
\mlabel{it:phi1}
\item $\varphi(\az(\fraku))=\az(\varphi(\fraku)).$
\mlabel{it:phi2}
\item
$\varphi(\fraku\ot \frakw)=\varphi(\fraku)\ot \az^{m-1}(w_1)\ot \cdots \ot\az^{m+n-2}(w_n)$.
\mlabel{it:phi3p}
\item
If $\ell(v)\geq 1$ and $\ell(u)=\ell(v)+1$, i.e., $m=k+1$,  then there exists  $\frakc\in \frakg^{\ot n}$ such that
$$\varphi(\fraku \ot \frakw)= \varphi(\fraku)\ot \frakc \quad \text{and}\quad \varphi(\frakv\ot \frakw)= \varphi(\frakv)\ot \az(\frakc).$$
\mlabel{it:phi3}
\end{enumerate}
\mlabel{lem:phi}
\end{lemma}
\begin{proof}
\noindent
(\mref{it:phi1}).
This follows from the definition of $\varphi$ since $\az^2=\id$.
\smallskip

\noindent
(\mref{it:phi2}).
By Item~(\mref{it:phi1}), we get
\begin{align*}
\varphi\big(\az(\fraku)\big)
=&\varphi\big(\az(u_1)\ot\az(u_2)\ot \cdots \ot \az(u_{n-2})\ot \az(u_{n-1})\ot \az(u_n)\big)
=\az(u_1)\ot \ot_{k=2}^m \az^{k-2} (\az(u_k)) \\
=& \az(u_1)\ot \ot_{k=2}^m \az^{k-1}(u_k)
=\az\big(u_1\ot \ot_{k=2}^m \az^{k-2}(u_k)\big)
=\az (\varphi(\fraku)).
\end{align*}

\noindent
(\mref{it:phi3p}).
By Item~(\mref{it:phi1}), we obtain
\begin{eqnarray*}
&&\varphi(\fraku\ot \frakw)\\
&=&\varphi(u_1\ot u_2\ot u_3\ot \cdots \ot u_m\ot w_1\ot\cdots\ot w_n)\\
&=&u_1\ot u_2\ot \az(u_3)\ot \cdots\ot \az^{m-2}(u_m)\ot \az^{m-1}(w_1)\ot \cdots \ot \az^{m+n-2}(w_n)\\
&=& \varphi(\fraku) \ot \az^{m-1}(w_1) \ot \cdots \ot \az^{m+n-2}(w_n).
\end{eqnarray*}

\noindent
(\mref{it:phi3}). In Item~(\mref{it:phi3p}), denote $\frakc:=\az^{m-1}(w_1) \ot \cdots \ot \az^{m+n-2}(w_n)$. Then $\frakc$ is in $\frakg^{\ot n}$ and Item~(\mref{it:phi3p}) becomes
$$\varphi(\fraku\ot \frakw)=\varphi(\fraku)\ot \frakc.$$
Since $k=m-1$, Item~(\mref{it:phi3p}) also gives
$$\varphi(\frakv\ot \frakw)
= \varphi(\frakv) \ot \az^{m-2}(w_1)\ot \cdots\ot \az^{m+n-3}(w_n)= \varphi(\frakv) \ot \az(\frakc).
$$
\end{proof}

We now prove Proposition~\mref{prop:ncideal} with some additional information.
\begin{prop} $\mathrm{(=Proposition~\mref{prop:ncideal})}$
Let $(\frakg,[\,,\,]_\frakg,\beta_\frakg)$ be an involutive Hom-Lie algebra. Let $\varphi:T^+_h(\frakg)\to T^+_h(\frakg)$ be as defined in Eq.~$($\mref{eq:defphi}$)$. Let $I_{\frakg,\beta}$ be the Hom-ideal of $T^+_h(\frakg)$ defined in Theorem~\mref{thm:isougd}. Let
\begin{equation}
J_{\frakg,\beta}:=
\sum_{i,j\geq0}\sum_{\substack{\fraka\in \frakg^{\ot i}\\\frakb\in \frakg^{\ot j}}} \sum_{x,y\in \frakg}  \bigg( \fraka\ot(x\ot y-y\ot x )\ot \frakb - \az(\fraka)\ot [x,y]_\frakg \ot \az(\frakb)\bigg).
\mlabel{eq:Jgb1}
\end{equation}
Then
\begin{enumerate}
\item
\begin{equation}
I_{\frakg,\beta}=\sum_{i,j\geq 0}\sum_{x,y\in \frakg} \Big(\frakg^{\ot i}\ohz (x\ot y-y\ot x -[x,y]_{\beta} )\Big)\ohz \frakg^{\ot j}.
\mlabel{eq:Ideal-Ic}
\end{equation}
\mlabel{it:ncpbw0}
\item
$\varphi(I_{\frakg,\beta})=J_{\frakg,\beta}.$
\mlabel{it:ncpbw1}
\end{enumerate}
\mlabel{prop:ncideal1}
\end{prop}
\begin{proof}
(\mref{it:ncpbw0}) 
Since the Hom-ideal $I_{\frakg,\beta}$ is generated by the elements of the form ~$x\ot y-y\ot x-[x,y]_\frakg, x,y\in \frakg,$
the right hand side is contained in the left hand side. To prove that the left hand side is contained in the right hand side, we just need to prove that the right hand side is a Hom-ideal of $T^+_h(\frakg)$, that is, it is closed under the left and right multiplication, and the operator $\az$. So we will check these conditions one by one.

For any nature number $k\geq 0$, we have
\begin{eqnarray*}
&&\Big(\big(\frakg^{\ot i}\ohz (x\ot y-y\ot x -[x,y]_{\frakg})\big)\ohz \frakg^{\ot j}\Big)\ohz \frakg^{\ot k}
\\
&=&\az\big(\frakg^{\ot i}\ohz (x\ot y-y\ot x -[x,y]_{\frakg})\big)\ohz (\frakg^{\ot j}\ohz \az(\frakg^{\ot k}))
\quad (\text{by Hom-associativity})
\\
&=&\az\big(\frakg^{\ot i}\ohz (x\ot y-y\ot x -[x,y]_{\frakg})\big)\ohz \frakg^{\ot j+k}
\quad (\text{by Eqs.~(\mref{eq:alphbn}) and (\mref{eq:otsum})})
\\
&=&\Big(\az(\frakg^{\ot i})\ohz \az(x\ot y-y\ot x -[x,y]_{\frakg})\Big)\ohz \frakg^{\ot j+k}
\quad (\text{by Eq.~(\mref{eq:ohcom})})
\\
&=&\Big(\frakg^{\ot i}\ohz \big(\az(x)\ot \az(y)-\az(y)\ot \az(x) -[\az(x),\az(y)]_{\frakg}\big)\Big)\ohz \frakg^{\ot j+k}
\quad (\text{by Eqs.~(\mref{eq:alpcom})\,and \,(\mref{eq:alphbn})})
\end{eqnarray*}
which is contained in
$\sum_{s,t\in \frakg} \big(\frakg^{\ot i}\ohz (s\ot t-t\ot s-[s,t]_{\frakg}) \big)\ohz \frakg^{\ot j+k})$ since $\az(\frakg)=\frakg$. Thus the right hand side of Eq.~(\mref{eq:Ideal-Ic}) is closed  under the right multiplication. A similar argument shows that it is also closed under the left multiplication.

By Eqs.~(\mref{eq:ohcom}), (\mref{eq:alpcom})\,and\,~(\mref{eq:alphbn}), we get
$$
\az((\frakg^{\ot i}\ohz (x\ot y-y\ot x [x,y]_{\frakg}))\ohz \frakg^{\ot j})
=(\frakg^{\ot i}\ohz(\az(x)\ot \az(y)-\az(y)\ot \az(x)-[\az(x),\az(y)]_{\frakg}))\ohz\frakg^{\ot j}$$
which is contained in $\sum_{s,t\in \frakg} (\frakg^{\ot i}\ohz (s\ot t-t\ot s -[s,t]_{\frakg}))\ohz \frakg^{\ot j}.$
So the right hand side of Eq.~(\mref{eq:Ideal-Ic}) is a Hom-ideal of $(T^+_h(\frakg),\ohz,\az)$ containing the elements of the form $x\ot y-y\ot x-[x,y]_{\frakg}, x,y\in \frakg$. Hence it contains the left hand side.
\smallskip

\noindent
(\mref{it:ncpbw1})
We first prove that $\varphi(I_{\frakg,\beta})$ is contained in $J_{\frakg,\beta}$. By Item~(\mref{it:ncpbw0}), we just need to verify that, for arbitrary $\fraka:=a_1\ot \cdots\ot a_i\in \frakg^{\ot i}$, $\frakb:=b_1\ot \cdots\ot
b_j\in\frakg^{\ot j}, i, j\geq 0$, and $x, y\in \frakg$, the element
$\varphi\big((\fraka \ohz(x\ot y-y\ot x-[x,y]_\frakg))\ohz \frakb\big)$ is contained in $J_{\frakg,\beta}$.

By the definition of $\ohz$ in Eqs.~(\mref{eq:ohini}) and (\mref{eq:ohrec}), we get
\begin{eqnarray}
&&\bigg(\fraka\ohz (x\ot y-y\ot x-[x,y]_\frakg)\bigg)\ohz \frakb\nonumber\\
&=&\bigg(\az(\fraka)\ot\Big(x\ot \az(y)-y\ot \az(x)\Big)-\fraka\ot [x,y]_\frakg\bigg)\ohz \frakb\nonumber\\
&=&\az^{j-1}\bigg(\az(\fraka)\ot\Big(x\ot \az(y)-y\ot \az(x)\Big)\bigg)\ot b_1\ot \az(b_2\ot \cdots \ot b_j)\nonumber\\
&&-\az^{j-1}\big(\fraka\ot [x,y]_\frakg\big)\ot b_1\ot \az(b_2\ot \cdots \ot b_j).\nonumber
\mlabel{eq:lefthom}
\end{eqnarray}
The two terms share the same right factor $w:=b_1\ot \az(b_2\ot \cdots \ot b_j)$ but the lengths of their left factors differ by one. So by Lemma~\mref{lem:phi}(\mref{it:phi3}), there exists $\frakc \in \frakg^{\ot i}$ such that
\begin{eqnarray*}
&&\varphi\bigg(\Big(\fraka\ohz(x\ot y-y\ot x-[x,y]_\frakg)\Big)\ohz \frakb\bigg)\\
&=& \varphi\bigg(\az^{j-1}\Big(\az(\fraka)\ot(x\ot \az(y)-y\ot \az(x))\Big)\bigg) \ot \frakc
-  \varphi\bigg( \az^{j-1}\Big(\fraka\ot[x,y]_\frakg\Big)\bigg) \ot \az(\frakc)\\
&=& \varphi\bigg(\az^{j}(\fraka)\ot\Big(\az^{j-1}(x)\ot \az^j(y)-\az^{j-1}(y)\ot \az^j(x)\Big)\bigg) \ot \frakc\\
&& -  \varphi\bigg( \az^{j-1}(\fraka)\ot[\az^{j-1}(x),\az^{j-1}(y)]_\frakg\bigg) \ot \az(\frakc) \quad (\text{by Eqs.~(\mref{eq:alpcom}) and (\mref{eq:hommult})})\\
&=& \varphi(\az^{j}(\fraka))\ot\Big(\az^{i-1}(\az^{j-1}(x))\ot \az^i(\az^j(y))-\az^{i-1}(\az^{j-1}(y))\ot \az^i(\az^j(x))\Big) \ot \frakc\\
&& -  \varphi( \az^{j-1}(\fraka))\ot\az^{i-1}([\az^{j-1}(x),\az^{j-1}(y)]_\frakg) \ot \az(\frakc) \quad (\text{by Lemma~\mref{lem:phi}(\mref{it:phi3p})})\\
&=& \varphi(\az^j(\fraka))\ot\Big(\az^{i+j}(x)\ot \az^{i+j}(y)-\az^{i+j}(y)\ot \az^{i+j}(x)\Big) \ot \frakc\\
&& -  \az(\varphi( \az^j(\fraka)))\ot[\az^{i+j}(x),\az^{i+j}(y)]_\frakg \ot \az(\frakc) \quad (\text{by Lemma~\mref{lem:phi}(\mref{it:phi2}), Eqs.~(\mref{eq:invol}) and (\mref{eq:hommult})}).
\end{eqnarray*}
This is an element in
$\sum_{\substack{\fraku\in \frakg^{\ot i}\\\frakv\in \frakg^{\ot j}}} \sum_{s,t\in \frakg}   \fraku\ot(s\ot t-t\ot s )\ot \frakv - \az(\fraku)\ot [s,t]_\frakg \ot \az(\frakv)$
by taking $\fraku:=\varphi(\az^j(\fraka)), \frakv:=\frakc, s:=\az^{i+j}(x)$ and $t:=\az^{i+j}(y)$. Thus $\varphi(I_{\frakg,\beta})$ is contained in $J_{\frakg,\beta}$. On the other hand, since $\varphi$ and $\az$ are bijective. The above argument shows that any term $\fraku\ot(s\ot t-t\ot s )\ot \frakv - \az(\fraku)\ot [s,t]_\frakg \ot \az(\frakv)$ in the previous sum can be expressed in the form $\varphi\bigg(\Big(\fraka\ohz(x\ot y-y\ot x-[x,y]_\frakg)\Big)\ohz \frakb\bigg)$. Thus $\varphi$ is surjective. This completes the proof.
\end{proof}

\subsection{A parameterized PBW decomposition}
\mlabel{ss:parpbw}

In this section, we take $\frakg$ to be an  involutive Hom-Lie algebra with an ordered basis $X$. So $X=\{x_i~|~i\in \omegaup\}$ for a linearly ordered set $\omegaup$.

\begin{theorem} $($=Theorem~\mref{thm:genpbw}$)$
Let $(\frakg,[\,,\,]_\frakg,\beta_\frakg)$ be an involutive Hom-Lie algebra such that $\beta_\frakg(X)=X$. Let $W$ be as defined in Eq.~$($\mref{eq:pbwbasis}$)$. Let $\mu\in \bfk$ be given.
Denote
\begin{equation}
J_{\frakg,\beta,\mu}:=
\sum_{i,j\geq0}\sum_{\substack{\fraka\in \frakg^{\ot i}\\\frakb\in \frakg^{\ot j}}} \sum_{x,y\in \frakg}  \bigg( \fraka\ot\Big((x\ot y-y\ot x )\ot \frakb\Big) - \mu^{i+j}\az(\fraka)\ot [x,y]_\frakg \ot \az(\frakb)\bigg).
\mlabel{eq:jgbeta1}
\end{equation}
Then we have the linear decomposition
\begin{equation}
T^+_h(\frakg)=J_{\frakg,\beta,\mu}\oplus \bfk W.
\end{equation}
\mlabel{thm:genpbw1}
\end{theorem}
We note that the sum in Eq.~(\mref{eq:jgbeta1}) remains the same when $x, y$ in the sum are taken from $X$.
\begin{proof}
We follow the proof of the classical PBW Theorem as presented in~\cite{Var}. We  need to prove that
$$T^+_h(\frakg)=J_{\frakg,\beta,\mu}+\bfk W \quad \text{and}\quad J_{\frakg,\beta,\mu}\bigcap \bfk W=0.$$

We first introduce some notations. Let~$n\geq 2$. Let $\frakx:=x_{i_1}\ot x_{i_2}\ot \cdots \ot x_{i_n}\in X^{\ot n}\subseteq \frakg^{\ot n}$. Define the index of $\frakx$ to be
$$d:=\ind(\frakx):=\big|~\{(r,s)~|~r<s~\text{and}~ i_r<i_s, 1\leq r,s\leq n\}~\big|. $$
Let~$ \frakg_{n,d}$ be the linear span of all pure tensors $\frakx$ of degree $n$ and index $d$. Then we obtain
\begin{equation}
\frakg^{\ot n}=\bigoplus_{d\geq 0}\frakg_{n,d}.
\mlabel{eq:gdn}
\end{equation}
In particular, $\frakg_{n,0}=\bfk W^{(n)}$, where $W^{(n)}:=\{x_{i_1}\ot x_{i_2}\ot \cdots\ot x_{i_n}\in X^{\ot n}\,|\, i_1\geq i_2\geq \cdots\geq i_n\}$.

To prove $T^+_h(\frakg)= J_{\frakg,\beta,\mu} +\bfk W$, by the equation ~$\bfk W=\sum_{i\geq 1}\bfk W^{(i)}$, we just need to prove that
\begin{equation}
\frakg^{\ot n}\subseteq J_{\frakg,\beta,\mu} + \sum_{1\leq q\leq n}\bfk W^{(q)}
\mlabel{eq:sum}
\end{equation}
by induction on $n\geq 1$.
For $n=1$, we have $\bfk W^{(1)}=\frakg$. So $\frakg\subseteq J_{\frakg,\beta,\mu} +\bfk W^{(1)}$.  Assume that Eq.~(\mref{eq:sum}) has been proved for $n\geq 1$.
Since $\frakg^{\ot (n+1)}=\sum_{d\geq 0}\frakg_{n+1,d}$, we just need to prove that
$$\frakg_{n+1,d}\subseteq J_{\frakg,\beta,\mu}+\sum_{1\leq q\leq n+1}\bfk W^{(q)}, \quad\text{for all}\,\, d\geq 0.$$
We accomplish this by induction on ~$d\geq 0$. For $d=0$, we have $\frakg_{n+1,0}=\bfk W^{(n+1)}$. Assume that, for $e\geq 0$, ~$\frakg_{n+1,e}\subseteq J_{\frakg,\beta,\mu} +\sum_{1\leq q\leq n+1} \bfk W^{(q)}$ has been proved.
Let $\frakx=x_{i_1}\ot x_{i_2}\ot\cdots\ot x_{i_{n+1}}\in X^{\ot (n+1)}\cap\frakg_{n+1,e+1}.$ Since $e+1\geq 1$, we can choose an integer $1\leq r\leq n$ such that $i_r<i_{r+1}$.  Define $\frakx'=x_{i_1}\ot \cdots
\ot x_{i_{r+1}}\ot x_{i_r}\ot\cdots \ot x_{i_{n+1}}$ to be the pure tensor formed by interchanging  $x_{i_r}$ with $x_{i_{r+1}}$ in $\frakd$. Then $\frakx'\in \frakg_{n+1,e}\subseteq J_{\frakg,\beta,\mu}+\sum_{1\leq q\leq n+1}\bfk W^{(q)}$. Since the definition of $J_{\frakg,\beta,\mu}$ gives
$$\frakx-\frakx'\equiv \mu^{n-1}\az(x_{i_1}\ot \cdots \ot x_{i_{r-1}})\ot [x_{i_r},x_{i_{r+1}}]_\frakg\ot \az(x_{i_{r+2}}\ot \cdots\ot x_{i_{n+1}})(\mathrm{mod}\, J_{\frakg,\beta,\mu}),$$
by the induction hypothesis on $n$, we have $\frakx\in J_{\frakg,\beta,\mu}+\sum_{1\leq q\leq n+1}\bfk W^{(q)}+\sum_{1\leq q\leq n}\bfk W^{(q)}$. So $\frakx$ is in $J_{\frakg,\beta,\mu}+\sum_{1\leq q\leq n+1}\bfk W^{(q)}$. This proves that $\frakg_{n+1,e+1}\subseteq J_{\frakg,\beta,\mu}+\sum_{1\leq q\leq n+1}\bfk W^{(q)}$. Together with the induction hypothesis on $d$, we have $\frakg^{\ot n+1}\subseteq J_{\frakg,\beta,\mu}+\sum_{1\leq q\leq n+1}\bfk W^{(q)}$. This completes the induction on $n$
\smallskip

We next prove that $J_{\frakg,\beta,\mu}\bigcap \bfk W=0$. We achieve this by constructing an operator ~$L$ on $T^+_h(\frakg)$ such that
\begin{equation}
\left\{\begin{array}{llll}
(1)& L(t)=t \,\,\text{for all}\, t \in W;&\\
(2)&\text{if}\, p \geq 2, 1\leq s\leq p-1, \text{and}\, i_s<i_{s+1}, \text{then}
&\\
&L(x_{i_1}\ot \cdots \ot x_{i_s}\ot x_{i_{s+1}}\ot \cdots \ot x_{i_p})=L(x_{i_1}\ot \cdots \ot x_{i_{s+1}}\ot x_{i_s}\ot \cdots \ot x_{i_p})&\\
&+L(\mu^{p-2}\az(x_{i_1}\ot \cdots\ot x_{s-1})\ot[ x_{i_s}, x_{i_{s+1}}]_\frakg\ot\az(x_{i_{s+2}}\ot \cdots \ot x_{i_p})).&
\end{array}\right.
\mlabel{eq:defnL}
\end{equation}
We define $L$ on $\sum_{1\leq q\leq n}\frakg^{\ot q}$ by induction on $n\geq 1$. For $n=1$, we define $L$ to be the identity on $\frakg$. Suppose that $n\geq 2$ and that $L$ is  an operator on ~$\sum_{1\leq q\leq n}\frakg^{\ot q}$ satisfying Eq.~(\mref{eq:defnL}) for all pure tensors of degree $n$. We will extend $L$ to an operator on $\sum_{1\leq q\leq n+1}\frakg^{\ot q}$ that satisfies Eq.~(\mref{eq:defnL}) for all pure tensors $\frakx= x_{i_1}\ot \cdots\ot x_{i_{n+1}}\in X^{\ot (n+1)}\subseteq \frakg^{\ot (n+1)}(=\sum_{d\geq 0}\frakg_{n+1,d})$.  For this we by induction on $d:=\ind(\frakx)$. For $d=0$, we define $L(\frakx)=\frakx$. Assume that $L(\frakx)$ has been defined for $\frakx\in \sum_{1\leq p\leq e}\frakg_{n+1,p}$, where $e\geq 0$. Consider $\frakx\in \frakg_{n+1,e+1}$. Choose an integer $r$, $1\leq r\leq n$, such that $i_r<i_{r+1}$. Then we define
$$L(\frakx):=L(x_{i_1}\ot \cdots\ot x_{i_{r+1}}\ot x_{i_r}\ot \cdots\ot x_{i_{n+1}})+L(\mu^{n-1}\az(x_{i_1}\ot \cdots \ot x_{i_{r-1}})\ot [x_{i_r},x_{i_{r+1}}]_\frakg\ot \az(x_{i_{r+2}}\ot \cdots\ot x_{i_{n+1}})).$$

We next verify that $L(\frakx)$ is well-defined, independent of the choice of the integer $1\leq r\leq n$ such that $i_r<i_{r+1}$. For this purpose, let $\ell$ be another integer, ~$1\leq \ell\leq n$, with $i_{\ell}<i_{\ell+1}$. Let
$$u:=L(x_{i_1}\ot \cdots\ot x_{i_{r+1}}\ot x_{i_r}\ot \cdots\ot x_{i_{n+1}})+L(\mu^{n-1}\az(x_{i_1}\ot \cdots \ot x_{i_{r-1}})\ot [x_{i_r},x_{i_{r+1}}]_\frakg\ot \az(x_{i_{r+2}}\ot \cdots\ot x_{i_{n+1}}))$$
and
$$v:=L(x_{i_1}\ot \cdots\ot x_{i_{\ell+1}}\ot x_{i_\ell}\ot \cdots\ot x_{i_{n+1}})+L(\mu^{n-1}\az(x_{i_1}\ot \cdots \ot x_{i_{\ell-1}})\ot [x_{i_\ell},x_{i_{\ell+1}}]_\frakg\ot \az(x_{i_{\ell+2}}\ot \cdots\ot x_{i_{n+1}})).$$
Then by the induction hypothesis, $u,v\in\sum_{0\leq p\leq e}\frakg_{n+1,p}+\sum_{1\leq q\leq n}\frakg^{\ot q}$ and satisfy Eq.~(\mref{eq:defnL}). We need to check ~$u=v$.
We distinguish two cases.
\smallskip

\noindent
{\bf Case 1: $|r-\ell|\geq 2$. } Then $r-\ell\geq 2$ or $\ell-r\geq 2$, and so $n\geq 3$.  Without lose of generality, we suppose $\ell-r \geq 2$. Since $u,v \in \sum_{0\leq p\leq e}\frakg_{n+1,p}+\sum_{1\leq q\leq n}\frakg^{\ot q}$, by the induction hypothesis, we have
\begin{eqnarray*}
u&=&L(x_{i_1}\ot \cdots\ot x_{i_{r+1}}\ot x_{i_r}\ot \cdots\ot x_{i_\ell}\ot x_{i_{\ell+1}}\ot \cdots\ot x_{i_{n+1}})\\
&&+L(\mu^{n-1}\az(x_{i_1}\ot \cdots \ot x_{i_{r-1}})\ot [x_{i_r},x_{i_{r+1}}]_\frakg\ot \az(\cdots\ot x_{i_\ell}\ot x_{i_{\ell+1}}\ot \cdots\ot x_{i_{n+1}}))\\
&=&L(x_{i_1}\ot \cdots\ot x_{i_{r+1}}\ot x_{i_r}\ot \cdots\ot x_{i_{\ell+1}}\ot x_{i_\ell}\ot \cdots\ot x_{i_{n+1}})\\
&&+L(\mu^{n-1}\az(x_{i_1}\ot \cdots\ot x_{i_{r+1}}\ot x_{i_r}\ot \cdots)\ot [x_{i_\ell}, x_{i_{\ell+1}}]_\frakg\ot \az(\cdots\ot x_{i_{n+1}}))\\
&&+L(\mu^{n-1}\az(x_{i_1}\ot \cdots )\ot [x_{i_r},x_{i_{r+1}}]_\frakg\ot \az(\cdots\ot x_{i_\ell}\ot x_{i_{\ell+1}}\ot \cdots\ot x_{i_{n+1}}))
\end{eqnarray*}
and
\begin{eqnarray*}
v&=&L(x_{i_1}\ot \cdots\ot x_{i_{r}}\ot x_{i_{r+l}}\ot \cdots\ot x_{i_{\ell+1}}\ot x_{i_\ell}\ot \cdots\ot x_{i_{n+1}})\\
&&+L(\mu^{n-1}\az(x_{i_1}\ot \cdots\ot x_{i_{r}}\ot x_{i_{r+l}}\ot\cdots )\ot [x_{i_\ell},x_{i_{\ell+1}}]_\frakg\ot \az(x_{i_{r+2}}\ot\cdots\ot x_{i_{n+1}}))\\
&=&L(x_{i_1}\ot \cdots\ot x_{i_{r+1}}\ot x_{i_r}\ot \cdots\ot x_{i_{\ell+1}}\ot x_{i_\ell}\ot \cdots\ot x_{i_{n+1}})\\
&&+L(\mu^{n-1}\az(x_{i_1}\ot \cdots)\ot [x_{i_{r}}, x_{i_{r+l}}]_\frakg\ot \az(\cdots\ot x_{i_{\ell+1}}\ot x_{i_\ell}\ot \cdots\ot x_{i_{n+1}}))\\
&&+L(\mu^{n-1}\az(x_{i_1}\ot \cdots\ot x_{i_{r}}\ot x_{i_{r+l}}\cdots )\ot [x_{i_\ell},x_{i_{\ell+1}}]_\frakg\ot \az( \cdots\ot x_{i_{n+1}}))
\end{eqnarray*}
Moreover,  since $\beta_\frakg(X)=X$ and $x_{i_r}\neq x_{i_{r+1}}$, $x_{i_\ell}\neq x_{i_{\ell+1}}$, we have $\az(x_{i_r}), \az(x_{i_{r+1}}), \az(x_{i_\ell}), \az(x_{i_{\ell+1}})\in X$ and $\az(x_{i_r})\neq \az(x_{i_{r+1}})$, $\az(x_{i_\ell})\neq \az(x_{i_{\ell+1}})$. Depending on whether or not $\az(x_{i_r})>\az(x_{i_{r+1}})$ and $\az(x_{i_\ell})>\az(x_{i_{\ell+1}})$, there are four cases  to consider. We just consider the case of
$\az(x_{i_r})>\az(x_{i_{r+1}})$ and $\az(x_{i_\ell})< \az(x_{\ell+1})$. The other cases are similar.  Then by the induction hypothesis, we obtain
\begin{eqnarray*}
&&L(\mu^{n-1}\az(x_{i_1}\ot \cdots\ot x_{i_{r+1}}\ot x_{i_r}\ot \cdots)\ot [x_{i_\ell}, x_{i_{\ell+1}}]_\frakg\ot \az(\cdots\ot x_{i_{n+1}}))\\
&=&L(\mu^{n-1}\az(x_{i_1}\ot\cdots\ot x_{i_r}\ot x_{i_{r+1}}\ot \cdots )\ot [x_{i_\ell}, x_{i_{\ell+1}}]_\frakg\ot \az(\cdots\ot x_{i_{n+1}}))\\
&&+L(\mu^{2n-3}(x_{i_1}\ot \cdots\ot [\az(x_{i_{r+1}}), \az(x_{i_r})]_\frakg\ot \cdots\ot [\az(x_{i_\ell}), \az(x_{i_{\ell+1}})]_\frakg\ot \cdots\ot x_{i_{n+1}})).
\end{eqnarray*}
and
\begin{eqnarray*}
&&L(\mu^{n-1}\az(x_{i_1}\ot \cdots )\ot [x_{i_r},x_{i_{r+1}}]_\frakg\ot \az(\cdots\ot x_{i_\ell}\ot x_{i_{\ell+1}}\ot \cdots\ot x_{i_{n+1}}))\\
&=&L(\mu^{n-1}\az(x_{i_1}\ot \cdots )\ot [x_{i_r},x_{i_{r+1}}]_\frakg\ot \az(\cdots\ot x_{i_{\ell+1}}\ot x_{i_\ell}\ot \cdots\ot x_{i_{n+1}}))\\
&&+L(\mu^{2n-3}(x_{i_1}\ot \cdots \ot [\az(x_{i_r}),\az(x_{i_{r+1}})]_\frakg\ot \cdots\ot[\az( x_{i_\ell}),\az( x_{i_{\ell+1}})]_\frakg\ot \cdots\ot x_{i_{n+1}})).
\end{eqnarray*}
Plugging these into the previous expressions of $u$, we get
\begin{eqnarray*}
u&=&L(x_{i_1}\ot \cdots\ot x_{i_{r+1}}\ot x_{i_r}\ot \cdots\ot x_{i_{\ell+1}}\ot x_{i_\ell}\ot \cdots\ot x_{i_{n+1}})\\
&&+L(\mu^{n-1}\az(x_{i_1}\ot\cdots\ot x_{i_r}\ot x_{i_{r+1}}\ot \cdots )\ot [x_{i_\ell}, x_{i_{\ell+1}}]_\frakg\ot \az(\cdots\ot x_{i_{n+1}}))\\
&&+L(\mu^{n-1}\az(x_{i_1}\ot \cdots )\ot [x_{i_r},x_{i_{r+1}}]_\frakg\ot \az(\cdots\ot x_{i_{\ell+1}}\ot x_{i_\ell}\ot \cdots\ot x_{i_{n+1}}))\\
&&+L(\mu^{2n-3}(x_{i_1}\ot \cdots\ot [\az(x_{i_{r+1}}), \az(x_{i_r})]_\frakg\ot \cdots\ot [\az(x_{i_\ell}), \az(x_{i_{\ell+1}})]_\frakg\ot \cdots\ot x_{i_{n+1}}))\\
&&+L(\mu^{2n-3}(x_{i_1}\ot \cdots \ot [\az(x_{i_r}),\az(x_{i_{r+1}})]_\frakg\ot \cdots\ot[\az( x_{i_\ell}),\az( x_{i_{\ell+1}})]_\frakg\ot \cdots\ot x_{i_{n+1}}))\\
&=&L(x_{i_1}\ot \cdots\ot x_{i_{r+1}}\ot x_{i_r}\ot \cdots\ot x_{i_{\ell+1}}\ot x_{i_\ell}\ot \cdots\ot x_{i_{n+1}})\\
&&+L(\mu^{n-1}\az(x_{i_1}\ot\cdots\ot x_{i_r}\ot x_{i_{r+1}}\ot \cdots )\ot [x_{i_\ell}, x_{i_{\ell+1}}]_\frakg\ot \az(\cdots\ot x_{i_{n+1}}))\\
&&+L(\mu^{n-1}\az(x_{i_1}\ot \cdots )\ot [x_{i_r},x_{i_{r+1}}]_\frakg\ot \az(\cdots\ot x_{i_{\ell+1}}\ot x_{i_\ell}\ot \cdots\ot x_{i_{n+1}}))\\
&& \qquad \qquad\quad(\text{by skew-symmetry}\,[x_{i_r}, x_{i_{r+1}}]_\frakg=- [x_{i_{r+1}},x_{i_r}]_\frakg)\\
&=&v.
\end{eqnarray*}

\noindent
{\bf Case 2: $|r-\ell|=1$.}  Without loss of generality, we assume $\ell=r+1$. Then $i_r<i_{r+1}<i_{r+2}$. By the induction hypothesis, we obtain
\begin{eqnarray*}
u&=& L(x_{i_1}\ot \cdots\ot x_{i_{r+1}}\ot x_{i_r}\ot x_{i_{r+2}}\ot \cdots\ot x_{i_{n+1}})\\
&&+L(\mu^{n-1}\az(x_{i_1}\ot \cdots \ot x_{i_{r-1}})\ot [x_{i_r},x_{i_{r+1}}]_\frakg\ot \az(x_{i_{r+2}}\ot \cdots\ot x_{i_{n+1}}))\\
&=& L(x_{i_1}\ot \cdots\ot x_{i_{r+1}}\ot x_{i_{r+2}}\ot x_{i_r}\ot \cdots\ot x_{i_{n+1}})\\
&&+ L(\mu^{n-1}\az(x_{i_1}\ot \cdots\ot x_{i_{r+1}})\ot [x_{i_r}, x_{i_{r+2}}]_\frakg\ot\az( \cdots\ot x_{i_{n+1}}))\\
&&+L(\mu^{n-1}\az(x_{i_1}\ot \cdots \ot x_{i_{r-1}})\ot [x_{i_r},x_{i_{r+1}}]_\frakg\ot \az(x_{i_{r+2}}\ot \cdots\ot x_{i_{n+1}}))\\
&=& L(x_{i_1}\ot \cdots\ot x_{i_{r+2}}\ot x_{i_{r+1}}\ot x_{i_r}\ot \cdots\ot x_{i_{n+1}})\\
&&+L(\mu^{n-1}\az(x_{i_1}\ot \cdots)\ot [x_{i_{r+1}},x_{i_{r+2}}]_\frakg\ot \az(x_{i_r}\ot \cdots\ot x_{i_{n+1}}))\\
&&+ L(\mu^{n-1}\az(x_{i_1}\ot \cdots\ot x_{i_{r+1}})\ot [x_{i_r}, x_{i_{r+2}}]_\frakg\ot\az( \cdots\ot x_{i_{n+1}}))\\
&&+L(\mu^{n-1}\az(x_{i_1}\ot \cdots \ot x_{i_{r-1}})\ot [x_{i_r},x_{i_{r+1}}]_\frakg\ot \az(x_{i_{r+2}}\ot \cdots\ot x_{i_{n+1}}).
\end{eqnarray*}
and
\begin{eqnarray*}
v&=&L(x_{i_1}\ot \cdots\ot x_{i_r}\ot x_{i_{r+2}}\ot x_{i_r+1}\ot \cdots\ot x_{i_{n+1}})\\
&&+L(\mu^{n-1}\az(x_{i_1}\ot \cdots \ot x_{i_r})\ot [x_{i_r+1},x_{i_{r+2}}]_\frakg\ot \az( \cdots\ot x_{i_{n+1}}))\\
&=&L(x_{i_1}\ot \cdots\ot x_{i_{r+2}}\ot x_{i_r}\ot x_{i_r+1}\ot \cdots\ot x_{i_{n+1}})\\
&&+L(\mu^{n-1}\az(x_{i_1}\ot \cdots)\ot [x_{i_r}, x_{i_{r+2}}]_\frakg\ot \az(x_{i_r+1}\ot \cdots\ot x_{i_{n+1}}))\\
&&+L(\mu^{n-1}\az(x_{i_1}\ot \cdots \ot x_{i_r})\ot [x_{i_r+1},x_{i_{r+2}}]_\frakg\ot \az( \cdots\ot x_{i_{n+1}}))\\
&=&L(x_{i_1}\ot \cdots\ot x_{i_{r+2}}\ot x_{i_r+1}\ot x_{i_r}\ot \cdots\ot x_{i_{n+1}})\\
&&+L(\mu^{n-1}\az(x_{i_1}\ot \cdots\ot x_{i_{r+2}})\ot [x_{i_r}, x_{i_r+1}]_\frakg\ot \az(\cdots\ot x_{i_{n+1}}))\\
&&+L(\mu^{n-1}\az(x_{i_1}\ot \cdots)\ot [x_{i_r}, x_{i_{r+2}}]_\frakg\ot \az(x_{i_r+1}\ot \cdots\ot x_{i_{n+1}}))\\
&&+L(\mu^{n-1}\az(x_{i_1}\ot \cdots \ot x_{i_r})\ot [x_{i_r+1},x_{i_{r+2}}]_\frakg\ot \az(\cdots\ot x_{i_{n+1}})).
\end{eqnarray*}
By the induction hypothesis, for any $x<y$, $t_1\in \frakg^{\ot m}, t_2\in \frakg^{\ot k}, m+k=n-2$, we have
$$L(t_1\ot x\ot y\ot t_2)-L(t_1\ot y\ot x \ot t_2)=L(\mu^{m+k}\az(t_1)\ot [x,y]\ot \az(t_2)).$$
So the sum of the last three terms of the previous expression of $u$ is
\begin{eqnarray*}
&&L(\mu^{n-1}\az(x_{i_1}\ot \cdots)\ot [x_{i_{r+1}},x_{i_{r+2}}]_\frakg\ot \az(x_{i_r}\ot \cdots\ot x_{i_{n+1}}))\\
&&+ L(\mu^{n-1}\az(x_{i_1}\ot \cdots\ot x_{i_{r+1}})\ot [x_{i_r}, x_{i_{r+2}}]_\frakg\ot\az( \cdots\ot x_{i_{n+1}}))\\
&&+L(\mu^{n-1}\az(x_{i_1}\ot \cdots \ot x_{i_{r-1}})\ot [x_{i_r},x_{i_{r+1}}]_\frakg\ot \az(x_{i_{r+2}}\ot \cdots\ot x_{i_{n+1}})\\
&=&L(\mu^{n-1}\az(x_{i_1}\ot \cdots)\ot \az(x_{i_r})\ot [x_{i_{r+1}},x_{i_{r+2}}]_\frakg\ot \az(\cdots\ot x_{i_{n+1}}))\\
&&+L(\mu^{2n-3}(x_{i_1}\ot \cdots\ot [[x_{i_{r+1}},x_{i_{r+2}}]_\frakg,\az(x_{i_r})]_\frakg\ot \cdots\ot x_{i_{n+1}}))\\
&&+ L(\mu^{n-1}\az(x_{i_1}\ot \cdots)\ot[x_{i_r}, x_{i_{r+2}}]_\frakg\ot \az(x_{i_{r+1}})\ot\az( \cdots\ot x_{i_{n+1}}))\\
&&+ L(\mu^{2n-3}x_{i_1}\ot \cdots\ot [\az(x_{i_{r+1}}), [x_{i_r}, x_{i_{r+2}}]_\frakg]_\frakg\ot\cdots\ot x_{i_{n+1}})\\
&&+L(\mu^{n-1}\az(x_{i_1}\ot \cdots \ot x_{i_{r-1}})\ot\az(x_{i_{r+2}})\ot [x_{i_r},x_{i_{r+1}}]_\frakg\ot\ot \az(\cdots\ot x_{i_{n+1}}))\\
&&+L(\mu^{2n-3}x_{i_1}\ot \cdots \ot x_{i_{r-1}}\ot [[x_{i_r},x_{i_{r+1}}]_\frakg,\az(x_{i_{r+2}})]_\frakg\ot \cdots\ot x_{i_{n+1}})\\
&=&L(\mu^{n-1}\az(x_{i_1}\ot \cdots\ot x_{i_r})\ot [x_{i_{r+1}},x_{i_{r+2}}]_\frakg\ot \az(\cdots\ot x_{i_{n+1}}))\\
&&+ L(\mu^{n-1}\az(x_{i_1}\ot \cdots)\ot[x_{i_r}, x_{i_{r+2}}]_\frakg\ot \az(x_{i_{r+1}}\ot \cdots\ot x_{i_{n+1}}))\\
&&+L(\mu^{n-1}\az(x_{i_1}\ot \cdots\ot x_{i_{r+2}})\ot [x_{i_r},x_{i_{r+1}}]_\frakg\ot\ot \az(\cdots\ot x_{i_{n+1}}))
\end{eqnarray*}
by the Hom-Jacobi identity. Thus
\begin{eqnarray*}
u&=&L(x_{i_1}\ot \cdots\ot x_{i_{r+2}}\ot x_{i_r+1}\ot x_{i_r}\ot \cdots\ot x_{i_{n+1}})\\
&&+L(\mu^{n-1}\az(x_{i_1}\ot \cdots\ot x_{i_r})\ot [x_{i_{r+1}},x_{i_{r+2}}]_\frakg\ot \az(\cdots\ot x_{i_{n+1}}))\\
&&+ L(\mu^{n-1}\az(x_{i_1}\ot \cdots)\ot[x_{i_r}, x_{i_{r+2}}]_\frakg\ot \az(x_{i_{r+1}}\ot \cdots\ot x_{i_{n+1}}))\\
&&+L(\mu^{n-1}\az(x_{i_1}\ot \cdots\ot x_{i_{r+2}})\ot [x_{i_r},x_{i_{r+1}}]_\frakg\ot\az(\cdots\ot x_{i_{n+1}}))\\
&=&v.
\end{eqnarray*}
Then we get $u=v$ in  both cases. So the assignment $L$ is a well-defined map.  Let $\frakx\in J_{\frakg,\beta}\bigcap \bfk W$. Then by the definition of $L$, we have $L(\frakx)=\frakx$ and $L(\frakx)=0$. Thus $\frakx=0$ and so
$J_{\frakg,\beta}\bigcap \bfk W=0$.

This completes the proof of Theorem~\mref{thm:genpbw1}.
\end{proof}

\noindent
{\bf Acknowledgments:} This work is supported by the National Natural Science Foundation of China (Grant No. 11071176, 11221101 and 11371178).

\end{document}